%% file: MOR.tex
\documentclass[3p, a4paper, 11pt]{elsarticle}

\input{default-el-journal.tex}
\usepackage{lineno}
\usepackage{math-symbols}
\usepackage{cases}
\usepackage{hyperref}
\usepackage{algorithm, algorithmicx, algpseudocode}
\usepackage{amsmath, amssymb}
\usepackage{graphicx}
\usepackage{booktabs}
\usepackage{tabularx}

\usepackage{epsfig}
\usepackage{epstopdf}

\graphicspath{{./images/}}

\journal{X} 

\begin{document}
\begin{frontmatter}
    \title{Riemannian optimization for model order reduction of linear systems with 
    quadratic outputs}
    \author[author1]{Xiaolong Wang}
    \author[author1]{Tongtu Tian}
    \address[author1]{School of Mathematics and Statistics, Northwestern Polytechnical University, Xi'an 710129, China}
    \begin{abstract}
        This paper investigates the optimal $H_2$ model order reduction for linear 
        systems with quadratic outputs. In the framework of Galerkin 
        projection, we first formulate the optimal $H_2$ MOR as an unconstrained 
        Riemannian 
        optimization problem on the Stiefel manifold. The Riemannian gradient of the 
        specific cost function is derived with the aid of Gramians of systems, and the 
        Dai-Yuan-type Riemannian conjugate gradient method is adopted to generate 
        structure-preserving reduced models. We also consider the optimal $H_2$ MOR based 
        on the product manifold, where some coefficient matrices of reduced models are 
        determined directly via the iteration of optimization problem, instead of the
        Galerkin projection method. In addition, we provide a scheme to compute low-rank 
        approximate solutions of Sylvester equations based on 
        the truncated polynomial expansions, which fully exploits the specific 
        structure of Sylvester equations in the optimization 
        problems, and enables an 
        efficient execution of our approach. Finally, two numerical examples are 
        simulated to  
        demonstrate the efficiency of our methods.
    \end{abstract}
\end{frontmatter}


\section{Introduction}\label{sec:introduction}

The large-scale dynamical systems appear frequently in various science and engineering 
applications, which are often formulated by a couple of differential equations. 
Simulation 
of such models is unacceptably time and storage consuming because of the high order 
of dynamical systems. Model order reduction (MOR) is a powerful tool to enable fast 
simulation of large-scale systems, which not only captures the dynamical 
behavior of systems accurately, but also reduces the 
computational load significantly during the simulation. Over the past few decades, MOR 
has been extensively studied, and there are various approaches used in the engineering, 
mainly including balanced truncation (BT) method, Krylov subspace methods, proper 
orthogonal decomposition (POD), as well as the
data-driven methods \cite{Benner2017, Antoulas2020}.

Structured systems are used to ensure physically meaningful response 
predictions in practice, and the special structure is often 
associated with the physics that are described by the systems, e.g., symmetries, time 
delays, and high-order time derivatives. There are many works that focus on MOR of the 
specific structured systems, such as port-Hamiltonian systems, second order systems, time 
delay systems and so on \cite{Benner20244, Beattie2009, Long2024}. In some applications, 
when the 
observed 
quantities are expressed as the variance or deviation of 
state variables from a reference point, the output equation of the dynamical systems
exhibits a quadratic structure, referred to systems with quadratic outputs. We consider 
the problem of MOR for linear quadratic outputs (LQO) systems in this paper. 
MOR techniques for LQO systems has been investigated in the past.
The most natural approach is to formulate LQO systems as linear systems with multiple 
outputs by the matrix 
decomposition, followed 
by the standard MOR methods for linear systems \cite{VanBeeumen2010, VanBeeumen2012}.
However, this approach often results in systems with a large number of outputs. The LQO 
system is formulated equivalently as a quadratic 
bilinear (QB) system in \cite{Pulch2019}, and then the BT method for QB 
systems is employed to produce reduced models. In order to enable a direct BT method for 
LQO systems, the Gramians and $H_2$ norm of LQO systems are defined 
in \cite{Benner2021}, and a structure-preserving BT method is developed 
along with an error estimation. The iterative rational Krylov algorithm has also been 
applied to LQO systems \cite{Gosea2019}. Based on the barycentric representations, the 
adaptive Antoulas-Anderson algorithm is extended to develop a data-driven modeling
framework for LQO systems \cite{Gosea2022}. More recently, $H_2$ optimal MOR of LQO 
systems is studied in \cite{Reiter2024}, and the Gramian-based first-order necessary 
conditions for reduced models of LQO systems are provided.

The optimal $H_2$ MOR for linear systems has been studied extensively. The Lyapunov- and 
interpolation-based conditions on the local optimality of reduced models are discussed in 
details in \cite{Gugercin2008}, where an iterative rational Krylov algorithm (IRKA) is 
designed 
to force optimality with respect to a set of interpolation conditions. As reduced models 
that satisfy the necessity 
of $H_2$ 
optimal conditions can be produced within the framework of projection, Riemannian 
optimization techniques are introduced into the field of MOR to minimize the $H_2$ error 
between the original and reduced models \cite{Yan1999,Xu2013,Yu2020}. Because of the 
existence 
of the minimum solution and convergence of Riemannian optimization problem, this approach 
is applied to MOR of other structured systems, such as second order systems, 
linear port-Hamiltonian systems, bilinear systems, and quadratic-bilinear systems 
\cite{Sato2017, Sato2017-2, Xu2015, Jiang2020-1}.  Combining the subspace iteration with 
MOR procedure, an online manifold learning approach is proposed for nonlinear 
dynamical systems in \cite{Peng2014}. The optimal $H_2$ MOR is also employed to 
construct a reduced stable 
positive network system with the preservation of the original interconnection structure 
in \cite{Sato2023-3}. Recently, IRKA is recast as a Riemannian gradient descent method
with a fixed step size over the manifold of rational functions
having fixed degree in \cite{Mlinaric2025}, and a more efficient execution is provided 
from the point of view.  For the theoretical analysis on the iterative algorithms over 
Riemannian manifolds, we refer the reader to \cite{Ring2012,Sato2016a,Huangwen2018}.

We investigate the optimal $H_2$ MOR problem of LQO systems based on Riemannian 
manifold. Note that an LQO system has a nonlinear input-output mapping, even though the 
dynamical equation is linear function of the states. The optimal $H_2$ MOR 
problem is considered first in the 
framework of Galerkin projection, and it can be characterized as an optimization problem 
over the Stiefel manifold. We adopt Riemannian conjugate gradient method to solve the 
related optimization problem, and the Riemannian gradient of the cost function is 
provided explicitly based on a couple of Sylvester equations. The resulting reduced 
models generated iteratively preserve 
the quadratic structure of original systems and are optimal in the sense of $H_2$ norm.  
We then relax slightly the projection methods, and select the coefficient matrices 
associated with the input and the output directly in Euclidean space. As a result, the 
optimization problem about the $H_2$ norm of error systems is described by the 
product 
manifold. Compared with the framework of Stiefel manifold, the larger feasible domain is 
defined by the product manifold, more accurate reduced models can be expected in this 
setting. For the execution of the proposed iterative algorithms, a couple of Sylvester 
equations are involved in each iterate, and solving these equations is time consuming in 
the  
large-scale settings. We provide low-rank approximate solutions to these Sylvester 
equations by fully exploiting their specific structure based on the truncated polynomial 
expansions. Consequently, we just need to solve the high order Sylvester equations only 
one time during the whole iteration, thereby enabling an efficient execution of our 
approach.

The paper is organized as follows. \autoref{sec:sec-2} introduces LQO systems and gives 
the preliminaries on 
Gramians and the $H_2$ norm. We start \autoref{sec:sec-3} with 
the 
formulation of the optimal $H_2$ MOR of LQO systems, and then establish Riemannian 
conjugate gradient method based on Steifel and product manifold, respectively, to 
generate reduced models iteratively, where the Riemannian gradient of the cost function 
is presented explicitly. A low-rank approximation to the solutions of Sylvester equations 
is provided in \autoref{sec:sec-4}.
Numerical examples are used to test our approach in \autoref{sec:sec-5}.
Finally, some conclusions are drawn in \autoref{sec:sec-6}.

\textit{Notation}: We assume that all matrices have compatible dimensions.
The notation $\mathbb{R}^{m\times n}$ represents the set of all $m\times n$ matrices with real entries.
For a matrix $A$, $A^{-1}$ denotes the inverse of $A$ when $A$ is a square matrix and is 
invertible, while $A^\top$ signifies the transpose of $A$. $\mathrm{S_n}$ represents the 
set of 
real symmetric matrices in $\mathbb{R}^{n\times n}$. The operator $\mathrm{vec}(M)$ for a 
given matrix $M$ produces a vector obtained by stacking the rows of $M$ one by one. The 
trace of a square matrix $M$ is denoted by $\mathrm{tr}(M)$. $\otimes$ denotes the 
Kronecker product of two matrices. $\mathrm{sym}(Z)$ denotes the symmetric part of the 
square 
matrix $Z$, that is $\mathrm{sym}(Z)=(Z+Z^{\top})/2$.


\section{Problem statement}\label{sec:sec-2}

We consider linear systems with quadratic outputs characterized by the following 
differential equations 
\begin{equation}\label{eqn-1}
	{\Sigma}: 
	\left\{\begin{array}{l}\dot{x}(t)=Ax(t)+Bu(t),\\y(t)=Cx(t)+x(t)^{\top}Mx(t),\end{array}\right.
\end{equation}
where $x\in\mathbb{R}^n$ is the state, $u(t)\in \mathbb{R}^m$ and $y(t)\in \mathbb{R}$ 
are 
the input and output functions, respectively. $A, M\in\mathbb{R}^{n\times n}$, 
$B\in\mathbb{R}^{n\times m}$, and $C\in\mathbb{R}^{1\times n}$ are the real coefficient
matrices of systems. We consider asymptotic stable LQO systems in this paper, that is, 
all eigenvalues of $A$ possess strictly negative real parts. We suppose that the number 
of input terminals is much smaller than the order of systems $\left(m\ll n\right)$. Note 
that $y(t)$ represents a quadratic output, and $\eqref{eqn-1}$ are 
multiple-input-single-output linear systems.
In general, there holds
\begin{equation}
	x^\top Mx=x^\top\left(\frac{1}{2}(M+M^\top)\right)x\quad\text{for 
	}\;x\in\mathbb{R}^n. \nonumber
\end{equation}
In what follows, we assume that $M=M^\top$ is a symmetric matrix, that is $M\in 
\mathrm{S_n}$. We 
aim to construct structure-preserving reduced models of (\ref{eqn-1}) as follows
\begin{equation}\label{eqn-2}
	\hat{\Sigma}:\left\{\begin{array}{l}\dot{\hat{x}}(t)=\hat{A} \hat{x}(t)+\hat{B} u(t), \\\hat{y}(t)=\hat{C}\hat{x}(t)+\hat{x}(t)^{\top}\hat{M}\hat{x}(t),\end{array}\right.
\end{equation}
where $\hat x(t)\in\mathbb{R}^{r}$, $u(t)\in\mathbb{R}^{m}$ and $\hat y(t)\in\mathbb{R}$ 
along with $r\ll n$. The dynamical behavior of (\ref{eqn-1}) should be approximated 
faithfully by (\ref{eqn-2}) for all admissible inputs $u(t)$, and its stability can 
also be preserved during the process of MOR.

Note that the relationship 
between the input and the state of (\ref{eqn-1}) is associated with linear time-invariant 
systems. The controllability Gramian of (\ref{eqn-1}) can be defined as 
\begin{equation}\label{eqn-3}
	P=\int_0^\infty e^{A\tau}BB^\top e^{A^\top\tau}\mathrm{d}\tau,
\end{equation}
which satisfies the following Lyapunov equation
\begin{equation}\label{eqn-4}
	AP+PA^\top+BB^\top=0.
\end{equation}
For the output in $\eqref{eqn-1}$, the linear part $Cx(t)$ corresponds to the 
observability Gramian
\begin{equation}
	Q_1 = \int_0^\infty e^{A^\top\sigma}C^\top Ce^{A\sigma}\mathrm{d}\sigma, \nonumber
\end{equation}
while the quadratic part ${x(t)}^\top Mx(t)$ corresponds to the observability Gramian
\begin{equation}
	Q_2=\int_0^\infty\int_0^\infty
	e^{A^\top\sigma_1}Me^{A\sigma_2}B\left(e^{A^\top\sigma_1}Me^{A\sigma_2}B\right)^\top
	\mathrm{d}\sigma_1\mathrm{d}\sigma_2, \nonumber
\end{equation}
which is referred as the quadratic-output observability 
Gramian in \cite{Gosea2019,Benner2021}.
As a result, the observability Gramian of (\ref{eqn-1}) is defined as $Q=Q_1+Q_2$, which 
is the unique solution of Lyapunov equation
\begin{equation}\label{eqn-7}
	A^\top Q+QA+C^\top C+MPM=0,
\end{equation}
where $P$ is the controllability Gramian, defined in $\eqref{eqn-4}$. Note that  $Q$ is a 
symmetric matrix due to $M\in\mathrm{S}_n$.

The $H_2$ norm of LQO systems $\eqref{eqn-1}$ is defined as
\begin{equation}
	\|\Sigma\|_{H_2}=\left(\int_0^\infty\|h_1(\sigma)\|_2^2\mathrm{d}\sigma+\int_0^\infty\int_0^\infty\|h_2(\sigma_1,\sigma_2)\|_2^2\mathrm{d}\sigma_1\mathrm{d}\sigma_2\right)^{\frac12},\nonumber
\end{equation}
where $h_1(\sigma)=Ce^{A\sigma}B$ and 
$h_2(\sigma_{1},\sigma_{2})=\mathrm{vec}\left(B^{\top}e^{A^{\top}\sigma_{1}}Me^{A\sigma_{2}}B\right)^{\top}$
are the linear and quadratic kernels, respectively.
It is well known that the $H_2$ norm of (\ref{eqn-1}) can be characterized via Gramians 
as follows
\begin{equation}
	\|\Sigma\|_{H_{2}}=\sqrt{\mathrm{tr}\left(B^{\top}QB\right)},\nonumber
\end{equation}
where $Q$ is the observability Gramian defined in $\eqref{eqn-7}$ 
\cite{Benner2021,Reiter2024}.
We employ the $H_2$ norm as a performance metric for the error induced by MOR of 
(\ref{eqn-1}). Note that the output error $y(t)-\hat{y}(t)$ at any time $t>0$ satisfies 
the inequality
\begin{equation}
	\|y(t)-\hat{y}(t)\|_{L_\infty}^2:=\sup_{t\geq0}\|y(t)-\hat{y}(t)\|_{\infty}\leq\|\Sigma-\hat{\Sigma}\|_{H_2}^2\left(\|u(t)\|_{L_2}^2+\|u(t)\otimes
	 u(t)\|_{L_2}^2\right).\nonumber
\end{equation}
This is to say, for an admissible input $u(t)$, a small $H_2$ error ensures that
the output $\hat y(t)$ of (\ref{eqn-2}) is a high-fidelity approximation
to that of the original systems in the $L_{\infty}$ sense. Consequently, we formulate MOR 
of LQO 
systems as an optimization problem over $H_2$ norm based on Riemannian 
manifold in the next section.

\section{Riemannian optimization for $H_2$ optimal MOR}\label{sec:sec-3}

We now discuss the $H_2$ optimal MOR problem for LQO systems. The $H_2$ norm of the error 
systems is viewed as a function of the coefficient matrices in (\ref{eqn-2}) as follows
\begin{equation}
	J(\hat{A},\hat{B},\hat{C},\hat{M})=\|\Sigma-\hat{\Sigma}\|_{H_2}^2.\nonumber
\end{equation}
We need the realization of the error system for computing the $H_2$ norm. In fact, the 
error 
system between $\eqref{eqn-1}$ and $\eqref{eqn-2}$ is defined by
\begin{equation}\label{eqn-8}
	{\Sigma_e}: (A_e,B_e,C_e,M_e)=
	\left(
	\left[ {\begin{array}{cc}
			A&0\\
			0&\hat{A}
	\end{array}} \right],
	\left[ {\begin{array}{cc}
		B\\\hat{B}
	\end{array}} \right],
	\left[ {\begin{array}{cc}
		C&-\hat{C}
	\end{array}} \right],
	\left[ {\begin{array}{cc}
		M&0\\
		0&-\hat{M}
	\end{array}} \right]
	\right), 
\end{equation}
and its output response reads $y_e=y-\hat{y}$.
Note that $\hat M$ is chosen as a symmetric matrix for reduced models and then 
$M_e\in \mathrm{S}_{n+r}$.
The controllability and observability Gramians $P_e$ and $Q_e$ of (\ref{eqn-8}) 
satisfy Lyapunov equations
\begin{align}
	A_eP_e+P_eA_e^\top+B_eB_e^\top&=0,\nonumber\\
	A_e^\top Q_e+Q_eA_e+C_e^\top C_e+M_eP_eM_e&=0, \nonumber
\end{align}
respectively. According to the structure of coefficient matrices defined in 
(\ref{eqn-8}), we 
partition $P_e$ and $Q_e$ into block forms
\begin{equation}\label{eqn-11}
	P_e=
	\left[ {\begin{array}{cc}
		P&X\\
		X^T&\hat{P}
	\end{array}} \right],\quad
	Q_e=
	\left[ {\begin{array}{cc}
		Q&Y\\
		Y^T&\hat{Q}
	\end{array}} \right].
\end{equation}
A straightforward algebraic manipulation leads to 
\begin{align}
	AX+X{\hat{A}^\top}+B{\hat{B}^\top}&=0,\label{eqn-12}\\
	\hat{A}\hat{P}+\hat{P}{\hat{A}^\top}+\hat{B}{\hat{B}^T}&=0,\label{eqn-13}\\
	A^\top Y+Y\hat{A}-C^\top\hat{C}-MX\hat{M}&=0,\label{Y_eqn}\\
	{\hat{A}^\top}\hat{Q}+\hat{Q}\hat{A}+\hat{C}^\top\hat{C}+\hat{M}\hat{P}\hat{M}&=0.\label{hat_Q_eqn}
\end{align}
Consequently, $J(\hat{A},\hat{B}, \hat{C}, \hat{M})$ can be expressed explicitly as
\begin{equation}
	J(\hat{A},\hat{B}, \hat{C}, 
	\hat{M})=\mathrm{tr}\left(B_e^TQ_eB_e\right)=\mathrm{tr}\left(B^\top 
	QB+2{B^\top}Y\hat{B}+{{\hat{B}}^\top}\hat{Q}\hat{B}\right),\nonumber
\end{equation}
where ${Q}$ is defined in $\eqref{eqn-7}$. We focus on reduced models (\ref{eqn-2}) that 
minimize the $H_2$ norm among the stable reduced models of order $r$. It can be 
formulated as a constrained optimization problem
\begin{equation}\label{eqn-16}
	\min_{{\hat{A}\in\mathbb{R}^{r\times r},\,\hat{A}\;is\;stable\atop\hat{B}\in\mathbb{R}^{r\times m},\,\hat{C}\in\mathbb{R}^{1\times r}}\atop\hat{M}\in S_r}J(\hat{A},\hat{B},\hat{C},\hat{M}).
\end{equation}

\subsection{Minimizing $H_2$-norm error over Stiefel manifold}\label{sec:sec-3.2}

The projection methods are well studied in the community of MOR, such as 
moment-matching, BT, POD and so on. As a special 
case, reduced models satisfying the first-order necessary conditions of $H_2$ optimal 
MOR can be produced in the framework of projection methods 
\cite{Gugercin2008,Reiter2024}. We construct reduced model (\ref{eqn-2}) of order $r$ 
such that it is ${H_2}$-norm optimal in the context of the orthogonal projection. For a 
given matrix $V\in\mathbb{R}^{n \times r}$ with orthogonal columns, i.e., $V^\top 
V=I_r$, the reduced models are determined by 
\begin{equation}\label{red_V}
	\hat{A}=V^\top AV,\ \hat{B}=V^\top B,\ \hat{C}=CV,\ \hat{M}=V^\top MV.
\end{equation}
Now (\ref{eqn-16}) boils down to the selection of a proper projection matrix 
$V\in\mathbb{R}^{n\times r}$ such that the $H_2$ norm error is as small as possible. We 
consider the set composed of all $n\times r$ column-orthogonal matrices
\begin{equation}
	\mathrm{St}(n,r)=\{V|V\in\mathbb{R}^{n\times r},V^\top V=I_r\}. \nonumber
\end{equation}
The set $\mathrm{St}(n,r)$ is the Stiefel manifold, also known as the compact or 
orthogonal Stiefel manifold. Note that the non-compact Stiefel manifold 
$\mathrm{St}_*(n,r)$ is the set of all 
full-rank matrices of order $n\times r$. Both of them are the embedded submanifolds of 
Euclidean space $\mathbb{R}^{n\times r}$. Although the dynamical behavior of 
(\ref{eqn-2}) 
remains unchanged after a state transformation, the projection matrix 
$V\in\mathrm{St}(n,r)$ is adopted typically in practice because of the superior numerical 
performance. The $H_2$ optimal MOR of $\eqref{eqn-1}$ can be characterized via Stiefel 
manifold as an unconstrained optimization problem 
\begin{equation}\label{eqn-19}
	\min_{V\in \mathrm{St}(n,r)}J_1(V)=J(V^\top AV,V^\top B,CV,V^\top MV).
\end{equation}
Note that we drop the constrain that $\hat A$ is stable. Because the $H_2$ norm of 
unstable LQO systems is infinity, the minimization of $H_2$ error ensures the stability 
of reduced models naturally.

In order to solve (\ref{eqn-19}) on Stiefel manifold via numerical methods, we need the 
basic notions of the tangent space, Riemannian metric and the gradient of $J_1(V)$.
Let $T_{V}\mathrm{St}(n,r)$ denote the tangent space to $\mathrm{St}(n,r)$ at $V\in \mathrm{St}(n,r)$.
The Stiefel manifold $\mathrm{St}(n,r)$ is a Riemannian manifold by endowing the tangent 
space $T_{V}\mathrm{St}(n,r)$ with the inner product
\begin{equation}
	\langle\xi_{1},\xi_{2}\rangle_V=\mathrm{tr}(\xi_{2}^\top\xi_{1}),\quad\forall\xi_{1},\xi_{2}\in T_{V}\mathrm{St}(n,r),\nonumber
\end{equation}
where the inner product on $T_{V}\mathrm{St}(n,r)$ is termed as the Riemannian metric, 
and induces a norm $\|\cdot\|_{V}$ on $T_{V}\mathrm{St}(n,r)$. 
The gradient of $J_1(V)$ at the point $V$ is referred as $\mathrm{grad}J_1(V)$, which is 
the unique tangent vector in $T_{V}\mathrm{St}(n,r)$ such that 
\begin{equation}
	\langle\xi,\text{grad}J_1(V)\rangle_V=\text{D}J_1(V)[\xi],\quad\forall\xi\in 
	T_V\text{St}(n,r),\nonumber
\end{equation}
where $\mathrm{D}J_1(V):T_{V}\mathrm{St}(n,r)\mapsto T_{J_1(V)}\mathbb{R}$ is the 
differential map of $J_1(V)$ at $V$. We use the continuation of $J_1(V)$ from Stiefel 
manifold 
$\mathrm{St}(n,r)$ to Euclidean space $\mathbb{R}^{n\times r}$, and define 
\begin{equation}\label{extension_J}
	\bar{J_1}(V)=\mathrm{tr}\big(B^\top 
	QB+2{B^\top}Y\hat{B}+{{\hat{B}}^\top}\hat{Q}\hat{B}\big),\quad 
	V\in\mathbb{R}^{n\times r}.
\end{equation}
It is obvious that $\bar{J_1}|_{\mathrm{St}(n,r)}=J_1$. Let $\mathrm{grad}\bar{J_1}(V)$ 
be the Euclidean gradient of $\bar{J_1}(V)$. The Riemannian gradient of $J_1(V)$ can be 
obtained by projecting $\bar{J_1}(V)$ onto $T_{V}\mathrm{St}(n,r)$, i.e.
\begin{equation}\label{eqn-22}
	\mathrm{grad}J_1(V)=P_V(\mathrm{grad}\bar{J_1}(V)), 
\end{equation}
where $P_V$ is the orthogonal projection onto the tangent space $T_{V}\mathrm{St}(n,r)$ 
\cite{Absil2008}
\begin{equation}
	P_V(D)=D-\frac12V(V^\top D+D^\top V),\quad D\in\mathbb{R}^{n\times r}.\nonumber
\end{equation}

When the linear-search strategy is used for a numerical 
solution to (\ref{extension_J}), the optimization 
variable is updated by $V_{j+1}=V_{j}+t_{j}\eta_{j}$ in the iteration, where $t_{j}$ is 
the step-size and $\eta_{j}$ is a proper direction. However, because $\mathrm{St}(n,r)$ 
is not a linear space, the 
linear-search method can not be applied directly to (\ref{eqn-19}). It is desirable to 
define a retraction map 
$\mathcal{R}_V:T_V\mathrm{St}(n,r)\rightarrow\mathrm{St}(n,r)$.
\newtheorem{definition}{Definition}
\newcommand{\defref}[1]{Definition~\ref{#1}}
\begin{definition}\label{def:1}
	A retraction on a manifold $\mathcal{M}$ is a smooth map 
	$\mathcal{R}:T\mathcal{M}\to\mathcal{M}\colon(x,v)\mapsto\mathcal{R}_x(v)$
	with the following properties:
	\\(i) $\mathcal{R}_x(0_x)=x$, where $0_x$ denotes the zero element of $T_x\mathcal{M}$.
	\\(ii) With the canonical identification $T_{0_x}T_x\mathcal{M}\simeq 
	T_x\mathcal{M}$, $\mathcal{R}_x$ satisfies 
	$D\mathcal{R}_x(0_x)={id}_{T_x\mathcal{M}},$ where ${id}_{T_x\mathcal{M}}$ denotes 
	the identity mapping on $T_x\mathcal{M}$.
\end{definition}
\par
The retraction map $\mathcal{R}_V$ transfers a vector in $T_V\mathrm{St}(n,r)$ to an 
element on $\mathrm{St}(n,r)$. 
Specifically, the retraction map of $\mathrm{St}(n,r)$ can be taken as
\begin{equation}\label{eqn-23}
	\mathcal{R}_V(\eta)=\mathrm{q}(V+\eta),\quad\eta\in T_V\mathrm{St}(n,r)
\end{equation}
where $\mathrm{q}(N)$ denotes the $\mathcal{Q}$ factor of the QR decomposition 
$N=\mathcal{Q}\mathfrak{R}$ for a given matrix
$N\in\mathbb{R}^{n\times r}$.  Note that $\mathcal{Q}\in\mathrm{St}(n,r)$ and 
$\mathfrak{R}$ is an upper triangular $n\times r$ matrix with 
strictly positive diagonal elements. We refer to \cite{Absil2008,Boumal2023} for more 
details 
about retraction map. The Riemannian conjugate gradient method is adopted in this paper 
to solve the optimization problem on manifold. We introduce the vector transport in order 
to update the search direction.
\begin{definition}\label{def:2}
	A vector transport on a manifold $\mathcal{M}$ is a smooth map
	\begin{equation}
		T\mathcal M\oplus T\mathcal M\to T\mathcal M:(\eta_x,\xi_x)\mapsto\mathcal T_{\eta_x}(\xi_x)\in T\mathcal M\nonumber
	\end{equation}
	satisfying the following properties for all $x\in\mathcal{M}$:
	\\(i) (Associated retraction) There exists a retraction $\mathcal{R}$, called the 
	retraction associated with $\mathcal{T}$, such that $\mathcal{T}_{\eta_x}(\xi_x)\in 
	T_{\mathcal{R}_x(\eta_x)}\mathcal{M}$ for $\xi_x,\eta_x\in T_x\mathcal{M}$;
	\\(ii) (Consistency) $\mathcal{T}_{0_x}(\xi_x)=\xi_x$ for $\xi_x\in T_x\mathcal{M}$;
	\\(iii) (Linearity) $\mathcal{T}_{\eta_x}(a\xi_x+b\zeta_x)=a\mathcal{T}_{\eta_x}(\xi_x)+b\mathcal{T}_{\eta_x}(\zeta_x)$. 
\end{definition}
\par
The vector transport is used to transform the elements in a tangent space to another one.
We adopt the following vector transport $\mathcal{T}$ on $\mathrm{St}(n,r)$ 
\begin{equation}
	\begin{aligned}
		\mathcal{T}_{\eta_V}(\xi_V)=&\mathrm{q}(V+\eta_V)\rho_{\mathrm{skew}}\Big({\mathrm{q}(V+\eta_V)}^\top \xi_V{\big({\mathrm{q}(V+\eta_V)}^\top(V+\eta_V)\big)}^{-1}\Big)\\
		&+\big(I_n-\mathrm{q}(V+\eta_V){\mathrm{q}(V+\eta_V)}^\top\big)\xi_V
		{\big({\mathrm{q}(V+\eta_V)}^\top(V+\eta_V)\big)}^{-1},\nonumber
	\end{aligned}
\end{equation}
where $V\in\mathrm{St}(n,r)$, $\xi_V, \eta_V\in T_V\mathrm{St}(n,r)$, and 
$\rho_{\mathrm{skew}}(D)$ returns a skew-symmetric matrix for a given matrix $D$, which 
is defined explicitly as 
\begin{equation}
	\left(\rho_{\text{skew}}(D)\right)_{i,j}=\begin{cases}D_{i,j}&\quad if\ i>j,\\0&\quad if\ i=j,\\-D_{j,i}&\quad if\ i<j,\end{cases}\qquad for\ D\in\mathbb{R}^{r\times r}.\nonumber
\end{equation}

Now we are in a position to employ the Riemannian conjugate gradient method to optimize 
the cost function defined in (\ref{eqn-19}).  Lemma 1 is helpful for the 
calculation of Riemannian gradient of $J_1(V)$. 
\newtheorem{lemma}{Lemma}
\newcommand{\lemmaref}[1]{Lemma~\ref{#1}}
\begin{lemma}\label{lemma:1}
	If $P$ and $Q$ satisfy $AP+PB+X=0$ and $A^\top Q+QB^\top+Y=0$, respectively,
	there holds $\mathrm{tr}(Y^\top P)=\mathrm{tr}(X^\top Q)$.
\end{lemma}
\begin{proof}
	It follows that $X=-(AP+PB), Y=-(A^\top Q+QB^\top)$. Due to the properties of the 
	trace function, it yields
	\begin{align}
		\mathrm{tr}(Y^\top P)&=-\mathrm{tr}\big(({A^\top Q+QB^\top})^\top P\big)\nonumber
		\\&=-\mathrm{tr}(Q^\top AP)-\mathrm{tr}(BQ^\top P)\nonumber
		\\&=-\mathrm{tr}\big((AP+PB)Q^\top\big)\nonumber
		\\&=\mathrm{tr}(XQ^\top)=\mathrm{tr}(Q^\top X),\nonumber
	\end{align}
	which concludes the proof. 
\end{proof}

The Riemannian gradient $\mathrm{grad}{J_1}(V)$ can be derived via the following theorem.
\newtheorem{theorem}{Theorem}
\begin{theorem}
	Consider LQO systems $\eqref{eqn-1}$ and reduced models $\eqref{eqn-2}$, which is 
	defined by (\ref{red_V}). If $A$ and $\hat A$ are stable matrices, the Riemannian 
	gradient of $J_1(V)$ with respect to $V\in \mathrm{St}(n, r)$ is expressed as 
	\begin{equation}\label{eqn-26}
		\mathrm{grad}J_1(V)=\mathrm{grad}\bar{J_1}(V)-\frac{1}{2}V\Big(V^T\mathrm{grad}\bar{J_1}(V)+\big(\mathrm{grad}\bar{J_1}(V)\big)^TV\Big),
	\end{equation}
	where the Euclidean gradient of $\bar{J_1}(V)$ with 
	respect to $V\in\mathbb{R}^{n\times r}$ is formulated as 
	\begin{align}\label{grad_J_E}
		\mathrm{grad}\bar{J_1}(V)=&2\big(A^\top V(X^\top K+\hat{P}L)^\top+AV(X^\top K+\hat{P}L)+B(\hat{B}^\top L+B^\top K)\nonumber\\
		&+C^\top(\hat{C}\hat{P}-CX)+2MV(\hat{P}\hat{M}\hat{P}-X^\top MX)\big),
	\end{align}
	in which $X$, $\hat{P}$ are determined by (\ref{eqn-12}) and (\ref{eqn-13}), and $K$, 
	$L$ 
	satisfy the Sylvester equations
	\begin{align}
		A^\top K+K\hat{A}-C^\top\hat{C}-2MX\hat{M}=&0,\label{eqn-27}\\
		\hat{A}^\top L+L\hat{A}+\hat{C}^\top\hat{C}+2\hat{M}\hat{P}\hat{M}=&0.\label{eqn-28}
	\end{align}
\end{theorem}

\begin{proof}
	For any matrix $\xi\in\mathbb{R}^{n\times r}$, differentiating both sides of 
	(\ref{extension_J}) leads to 
	\begin{align}
		\mathrm{D}\bar{J_1}(V)[\xi]&=\mathrm{tr}(2B^\top\mathrm{D}Y[\xi]\hat{B}+2B^\top Y{\xi}^\top B+B^\top\xi\hat{Q}\hat{B}+{\hat{B}}^\top\mathrm{D}\hat{Q}[\xi]\hat{B}+{\hat{B}}^\top\hat{Q}{\xi}^\top B)\nonumber\\
		&=2\mathrm{tr}\big(\xi^\top(BB^\top Y+B{\hat{B}}^\top\hat{Q})\big)+2\mathrm{tr}(B^\top\mathrm{D}Y[\xi]\hat{B})+\mathrm{tr}({\hat{B}}^\top\mathrm{D}\hat{Q}[\xi]\hat{B}),\label{eqn-29}
	\end{align}
	where $\mathrm{D}Y[\xi]$ and $\mathrm{D}\hat{Q}[\xi]$ is obtained by differentiating 
	(\ref{Y_eqn}) and (\ref{hat_Q_eqn}), respectively. It yields 
	\begin{align}
		A^\top\mathrm{D}Y[\xi]+\mathrm{D}Y[\xi]\hat{A}+N_1&=0,\label{eqn-30}\\
		\hat{A}^\top\mathrm{D}\hat{Q}[\xi]+\mathrm{D}\hat{Q}[\xi]\hat{A}+N_2&=0,\label{eqn-31}
	\end{align}
	in which
	\begin{align}
		N_1=&Y\xi^\top AV+YV^\top A\xi-C^\top C\xi-M\mathrm{D}X[\xi]\hat{M}-MX\xi^\top MV-MXV^\top M\xi,\nonumber\\
		N2=&\xi^\top A^\top V\hat{Q}+V^\top A^\top\xi\hat{Q}+\hat{Q}\xi^\top AV+\hat{Q}V^\top A\xi+{\xi}^\top C^\top\hat{C}+{\hat{C}}^\top C\xi+\hat{M}\mathrm{D}\hat{P}[\xi]\hat{M}\nonumber\\
		&+\xi^\top MV\hat{P}\hat{M}+V^\top M\xi\hat{P}\hat{M}+\hat{M}\hat{P}\xi^\top 
		MV+\hat{M}\hat{P}V^\top M\xi. \nonumber
	\end{align}
	Similarly, we differentiate (\ref{eqn-12}) and (\ref{eqn-13}) to derive 
	$\mathrm{D}X[\xi]$ and $\mathrm{D}\hat{P}[\xi]$ contained in the 
	above expression. There hold 
	\begin{equation}\label{eqn-32}
	A\mathrm{D}X[\xi]+\mathrm{D}X[\xi]\hat{A}^\top+X\xi^\top A^\top V+XV^\top A^\top\xi+BB^T\xi=0,
	\end{equation}
	\begin{equation}\label{eqn-33}
	\hat{A}\mathrm{D}\hat{P}[\xi]+\mathrm{D}\hat{P}[\xi]\hat{A}^\top+R=0,
	\end{equation}
	where 
	\begin{equation}
		R=\xi^\top AV\hat{P}+V^\top A\xi\hat{P}+\hat{P}\xi^\top A^\top V+\hat{P}V^\top A^\top\xi+\xi^\top B\hat{B}^\top+\hat{B}B^\top\xi.\nonumber
	\end{equation}

	 Consider Sylvester equations $\eqref{eqn-12}$ and $\eqref{eqn-30}$. It follows from 
	 \lemmaref{lemma:1} that 
	\begin{align}
		\mathrm{tr}(B^\top\mathrm{D}Y[\xi]\hat{B})=&\mathrm{tr}((B\hat{B}^\top)^\top\mathrm{D}Y[\xi])=\mathrm{tr}(X^\top N_1)\nonumber\\
		=&\mathrm{tr}\big(\xi^\top(AVX^\top Y+A^\top VY^\top X-C^\top CX-2MVX^\top MX)\big)\nonumber\\&-\mathrm{tr}(X^\top M\mathrm{D}X[\xi]\hat{M}).\label{eqn-34}
	\end{align}
	With the auxiliary Sylvester equation
	\begin{equation}\label{eqn-35}
		A^\top R_1+R_1\hat{A}+MX\hat{M}=0, 
	\end{equation}
	we apply \lemmaref{lemma:1} to $\eqref{eqn-32}$ and $\eqref{eqn-35}$
	\begin{align}
		\mathrm{tr}(X^\top M\mathrm{D}X[\xi]\hat{M})&=\mathrm{tr}\big((MX\hat{M})^\top\mathrm{D}X[\xi]\big)\nonumber\\
		&=\mathrm{tr}\big(R_1^\top(X\xi^\top A^\top V+XV^\top A^\top\xi+BB^T\xi)\big)\nonumber\\
		&=\mathrm{tr}\big(\xi^\top(AVX^\top R_1+A^\top VR_1^\top X+BB^\top R_1)\big), 
		\nonumber
	\end{align}
	and then $\eqref{eqn-34}$ boils down to
	\begin{equation}\label{eqn-37}
	    \begin{aligned}
	        \mathrm{tr}(B^\top\mathrm{D}Y[\xi]\hat{B})=&\mathrm{tr}\big(\xi^\top(AVX^\top(Y-R_1)+A^\top V(Y-R_1)^\top X\\
	        &-BB^\top R_1-C^\top CX-2MVX^\top MX)\big).
	    \end{aligned}
	\end{equation}
	We use the notation $K=Y-R_1$, which satisfies 
	\begin{equation}
		A^\top K+K\hat{A}-C^\top\hat{C}-2MX\hat{M}=0,\nonumber
	\end{equation}
	and $\eqref{eqn-37}$ is reformulated as
	\begin{equation}\label{eqn-38}
		\begin{aligned}
		    \mathrm{tr}(B^\top\mathrm{D}Y[\xi]\hat{B})=&\mathrm{tr}\big(\xi^\top(AVX^\top K+A^\top VK^\top X\\
		    &-BB^\top (Y-K)-C^\top CX-2MVX^\top MX)\big).
		\end{aligned}
	\end{equation}

	Similarly, by applying \lemmaref{lemma:1} to $\eqref{eqn-13}$, $\eqref{eqn-31}$, 
	$\eqref{eqn-33}$, one can validate that 
	\begin{equation}\label{eqn-39}
		\begin{aligned}
		    \mathrm{tr}({\hat{B}}^\top\mathrm{D}\hat{Q}[\xi]\hat{B})=
		    &2\mathrm{tr}(\xi^\top(AV\hat{P}L+A^\top VL\hat{P}\\
		    &+B\hat{B}^\top R_2+C^\top\hat{C}\hat{P}+2MV\hat{P}\hat{M}\hat{P})),
		\end{aligned}
	\end{equation}
	where $L$ satisfies \eqref{eqn-28} and $R_2$ solves the Sylvester equation 
	\begin{equation}
		\hat{A}^\top R_2+R_2\hat{A}+\hat{M}\hat{P}\hat{M}=0.\label{R_equation}
	\end{equation}
	Substituting $\eqref{eqn-38}$ and $\eqref{eqn-39}$ into $\eqref{eqn-29}$ gives
	\begin{align}
		\mathrm{D}\bar{J_1}(V)[\xi]=&2\mathrm{tr}\big(\xi^\top(AV(X^\top K+\hat{P}L)+A^\top V(X^\top K+\hat{P}L)^\top+B(\hat{B}^\top L+B^\top K)\nonumber\\
		&+C^\top(\hat{C}\hat{P}-CX)+2MV(\hat{P}\hat{M}\hat{P}-X^\top MX))\big).\nonumber
	\end{align}
    Because of  
	$\mathrm{D}\bar{J_1}(V)[\xi]=\mathrm{tr}(\xi^\top\mathrm{grad}\bar{J_1}(V))$, the 
	above formula of $\mathrm{D}\bar{J_1}(V)[\xi]$ implies the gradient 
	$\mathrm{grad}\bar{J_1}(V)$ given in (\ref{grad_J_E}). Finally, the Riemannian 
	gradient $\mathrm{grad}J_1(V)$ is derived by using the projection defined in  
	$\eqref{eqn-22}$. This completes the proof.
\end{proof}

We use Dai-Yuan-type Riemannian conjugate gradient method provided in 
\cite{Sato2016a} to optimize the $H_2$ norm of reduced models. Let $V_k$ be a current 
point 
on $\mathrm{St}(n,r)$. The next point in the iteration can be obtained via the retraction 
map 
\begin{equation}
	V_{k+1}=\mathcal{R}_{V_k}(t_k\eta_k),\nonumber
\end{equation}
where $\eta_k$  is the conjugate gradient direction along with a proper step size 
$t_k>0$ for $k=0,1, \cdots$. The search direction $\eta_{k}$ at the current iterate in 
Euclidean space is 
updated typically via 
\begin{equation*}
	\eta_{k}=-\mathrm{grad}J_1(V_{k})+\beta_{k}\eta_{k-1}.
\end{equation*}
For the optimization problem on Stiefel manifold, the search direction can be derived 
with the vector transport $\mathcal{T}$ defined in \defref{def:2} as follows 
\begin{equation}
	\eta_{k}=-\mathrm{grad}J_1(V_{k})+\beta_{k}\mathcal{T}_{t_{k-1}\eta_{k-1}}(\eta_{k-1})\nonumber
\end{equation}
with Dai–Yuan-type parameter
\begin{equation}\label{eqn-43}
	\beta_{k}=\frac{\|\mathrm{grad}J_1(V_{k})\|_{V_{k}}^2}{\langle\mathrm{grad}J_1(V_{k}),
		\mathcal{T}_{t_{k-1}\eta_{k-1}}(\eta_{k-1})\rangle_{V_{k}}-\langle 
		\mathrm{grad}J_1(V_{k-1}),\eta_{k-1}\rangle_{V_{k-1}}},
\end{equation}
for $k=1,2,\cdots$.
In practice the deflated vector 
transport $\tilde{\mathcal{T}}$, instead of $\mathcal{T}$, is used for the selection of 
$\eta_k$
\begin{equation}\label{eqn-44}
	\tilde{\mathcal{T}}_{t_{k-1}\eta_{k-1}}(\eta_{k-1})=
	\min\left\{1,\frac{\|\eta_{k-1}\|_{V_{k-1}}}{\|\mathcal{T}_{t_{k-1}\eta_{k-1}}
		(\eta_{k-1})\|_{V_{k}}}\right\}\mathcal{T}_{t_{k-1}\eta_{k-1}}(\eta_{k-1}),
\end{equation}
which guarantees the inequality
\begin{equation}
	\|\tilde{\mathcal{T}}_{\alpha_{k-1}\eta_{k-1}}(\eta_{k-1})\|_{V_{k}}\leq\|\eta_{k-1}\|_{V_{k-1}}\nonumber
\end{equation}
and thereby leads to the convergence of Riemannian conjugate gradient methods. Note that 
$\beta_0=0$ and $\eta_{0}=-\mathrm{grad}J_1(V_{0})$ for the initial direction. We use
Wolfe conditions to select the step size $t_{k}=\omega^{m_{k}}\gamma$, where 
$\gamma>0$, 
$\omega\in(0,1)$, and $m_k$ is the smallest non-negative integer satisfying
\begin{equation}\label{eqn-45}
	J_1(\mathcal{R}_{V_k}(\omega^{m_k}\gamma\eta_k))\leq J_1(V_k)+c_1\omega^{m_k}\gamma\langle\mathrm{grad}J_1(V_k),\eta_k\rangle_{V_k},
\end{equation}
\begin{equation}\label{eqn-46}
	\langle\mathrm{grad}J_1(\mathcal{R}_{V_k}(\omega^{m_k}\gamma\eta_k)),
	\mathcal{T}_{t_k\eta_k}(\eta_k)\rangle_{\mathcal{R}_{V_k}(\omega^{m_k}\gamma\eta_k)}\geq
	 c_2\langle\mathrm{grad}J_1(V_k),\eta_k\rangle_{V_k},
\end{equation}
with $0<c_1<c_2<1$. The iterative algorithm for the optimization problem (\ref{eqn-19}) 
the is presented in Algorithm 1.

\begin{algorithm}
	\renewcommand{\algorithmicrequire}{\textbf{Input:}}
	\renewcommand{\algorithmicensure}{\textbf{Output:}}
	\caption{Riemannian conjugate gradient method for MOR based on $\mathrm{St}(n,r)$ 
	(SRCG).}
	\label{alg:1}
	\begin{algorithmic}
		\Require The coefficient matrices of LQO system $\Sigma$, and a positive integer 
		$k_{max}$. 
		\Ensure Reduced LQO models $\hat\Sigma$.
		\\Choose a proper initial projection matrix $V_0\in\mathrm{St}(n,r)$.
		\\Compute the Riemannian gradient $\mathrm{grad}J_1(V_0)$ by $\eqref{eqn-26}$, 
		and set $\eta_0=-\mathrm{grad}J_1(V_0)$.
		\For{$k=0,1,\cdots, k_{max}-1$}
		\\\quad Choose a step size $t_k$ satisfying $\eqref{eqn-45}$ and $\eqref{eqn-46}$ 
		.
		\\\quad Set $V_{k+1}=\mathcal{R}_{V_k}(t_k\eta_k)$, and compute $\mathrm{grad}J_1(V_{k+1})$.
		\\\quad Compute $\tilde {\mathcal{T}}_{t_k\eta_k}(\eta_k)$ and $\beta_{k+1}$ by 
		$\eqref{eqn-44}$ and $\eqref{eqn-43}$, respectively.
		\\\quad Set 
		$\eta_{k+1}=-\mathrm{grad}J_1(V_{k+1})+\beta_{k+1}\tilde 
		{\mathcal{T}}_{t_k\eta_k}(\eta_k)$.
		\EndFor
		\\\textbf{return} $\hat{A}=V_{k_{max}}^\top AV_{k_{max}}$, 
		$\hat{B}=V_{k_{max}}^\top B$, 
		$\hat{C}=CV_{k_{max}}$, $\hat{M}=V_{k_{max}}^\top MV_{k_{max}}.$
	\end{algorithmic}  
\end{algorithm}

In theory, the iteration in Algorithm 1 should be executed until the local optimality 
conditions are 
fulfilled. As pointed out in Chapter 4 of \cite{Boumal2023}, any local 
minimizer $V\in 
\mathrm{St}(n,r)$ of the cost function is a critical point on which the norm of 
Riemannian 
gradient is zero, that is $\|\mathrm{grad}J_1(V)\|_{V}=0$. However, one can terminate 
the iteration in practice if the 
objective function in \eqref{eqn-19} decreases sufficiently or becomes small enough, or 
simply stops when the number of the iteration reaches a prespecified bound $k_{max}$.

Note that $\mathrm{tr}(B^\top QB)$ is a constant in the optimization problem 
\eqref{eqn-19}, which can be dropped directly in practice to avoid the calculation of the 
observability Gramian $Q$. In 
addition, 
the calculation of $\mathrm{grad}J_1(V_{k+1})$ in each iterate involves a couple of 
Sylvester equations \eqref{eqn-12} \eqref{eqn-13} \eqref{eqn-27} and \eqref{eqn-28}, 
which dominates the whole cost of \autoref{alg:1}. A scheme based on the polynomial 
expansion 
will be provided in next section to derive the approximate solution of the related 
Sylvester equations, which reduces the cost of the proposed algorithm dramatically.

\subsection{Extension to optimization problem over the product 
manifold}\label{sec:sec-3.4}

We have presented Algorithm 1 to optimize the $H_2$ norm of error systems in the 
framework of Galerkin approach. In this subsection we relax the projection 
methods slightly and optimize the coefficient matrices associated with the input and the 
output 
functions directly. Specifically, we assume $\hat{A}=U^\top AU$ in \eqref{eqn-2}, which 
is determined by $U\in\mathrm{St}(n,r)$ with
orthogonal columns, and select $\hat{B}$, $\hat{C}$ and $\hat{M}$ directly in Euclidean 
space, instead of restricting them into the framework of projection methods. In this 
settings, the $H_2$ optimal MOR problem $\eqref{eqn-16}$ can be described using the 
product manifold. Note that $\mathbb{R}^{r\times m}$, $\mathbb{R}^{1\times r}$ and 
$\mathrm{S_r}$ 
are all linear spaces, possessing a natural linear manifold structure. We employ the  
product manifold
\begin{equation}
	\mathcal{N}=\mathrm{St}(n,r)\times\mathbb{R}^{r\times m}\times\mathbb{R}^{1\times 
	r}\times \mathrm{S_r}, \nonumber
\end{equation}
and the $H_2$ optimal MOR problem \eqref{eqn-16} is characterized as follows 
\begin{equation}\label{eqn-47}
	\min_{(U,\hat{B},\hat{C},\hat{M})\in\mathcal{N}}\{J_2(U,\hat{B},\hat{C},\hat{M})=J(U^\top
	 AU,\hat{B},\hat{C},\hat{M})\}.
\end{equation}
It is clear that the feasible domain of $\eqref{eqn-47}$ is lager than that of  
$\eqref{eqn-19}$, and there holds 
\begin{equation}
	\mathrm{min}\{J_{2}\}\leq \mathrm{min}\{J_{1}\}.\nonumber
\end{equation}

We define the Riemannian metric for the product manifold $\mathcal{N}$
\begin{equation}
	\begin{split}
	&\langle(U_1^\prime,B_1^\prime,C_1^\prime,M_1^\prime),(U_2^\prime,B_2^\prime,C_2^\prime,M_2^\prime)
	\rangle_{(U,\hat{B},\hat C,\hat{M})}\\
	=&\mathrm{tr}({U_1^\prime}^\top 
	U_2^\prime)+\mathrm{tr}({B_1^\prime}^\top B_2^\prime)+\mathrm{tr}({C_1^\prime}^\top 
	C_2^\prime)+\mathrm{tr}({M_1^\prime}^\top M_2^\prime)\nonumber
	\end{split}
\end{equation}
for 
$(U_1^\prime,B_1^\prime,C_1^\prime,M_1^\prime),(U_2^\prime,B_2^\prime,C_2^\prime,M_2^\prime)\in
 T_{(U,\hat{B},\hat{C},\hat{M})}\mathcal{N}$, where 
$T_{(U,\hat{B},\hat{C},\hat{M})}\mathcal{N}$ denotes the tangent space of $\mathcal{N}$ 
at the point $(U,\hat{B},\hat{C},\hat{M})$.
Obviously, $\mathcal{N}$ can be regarded as a Riemannian submanifold of the linear space 
$\bar{\mathcal{N}}=\mathbb{R}^{n\times r}\times\mathbb{R}^{r\times 
m}\times\mathbb{R}^{1\times 
r}\times\mathbb{R}^{r\times r}$ along with a natural inner product \cite{Absil2008}.
The linear spaces $\mathbb{R}^{r\times m}$, $\mathbb{R}^{1\times r}$ and $\mathrm{S_r}$ 
possess a linear manifold structure, and the tangent spaces satisfy 
\begin{equation}\label{eqn-48}
	T_{\hat{B}}\mathbb{R}^{r\times m}\simeq\mathbb{R}^{r\times m},\quad 
	T_{\hat{C}}\mathbb{R}^{1\times r}\simeq\mathbb{R}^{1\times r}\quad\text{and}\quad 
	T_{\hat{M}}\mathrm{S_r}\simeq \mathrm{S_r}.
\end{equation}
As a consequence, the orthogonal projection from $\bar{\mathcal{N}}$ onto 
$T_{(U,\hat{B},\hat{C},\hat{M})}\mathcal{N}$ at the point $(U,\hat{B},\hat{C},\hat{M})\in 
\mathcal{N}$ 
is defined by 
\begin{equation}
	P_{(U,\hat{B},\hat{C},\hat{M})}(\bar{U},\bar{B},\bar{C},\bar{M})=
	(P_U(\bar{U}),\bar{B},\bar{C},\mathrm{sym}(\bar{M})),\nonumber
\end{equation}
where $(\bar{U},\bar{B},\bar{C},\bar{M})\in\bar{\mathcal{N}}$,
and $P_U$ is defined by 
\begin{equation}\label{pro_U}
	P_U(\bar{U})=\bar{U}-U(U^\top\bar{U}+\bar{U}^\top 
U)/2.
\end{equation} 
It follows that a retraction 
$\mathcal{R}_{(U,\hat{B},\hat{C},\hat{M})}$ on $\mathcal{N}$ reads
\begin{equation}
	\begin{aligned}
		\mathcal{R}_{(U,\hat{B},\hat{C},\hat{M})}(U^\prime,B^\prime,C^\prime,M^\prime)=
		&(\mathrm{q}(U+U^\prime),\hat{B}+B^\prime,\hat{C}+C^\prime,\hat{M}+M^\prime)\nonumber
	\end{aligned}
\end{equation}
for $(U^\prime,B^\prime,C^\prime,M^\prime)\in 
T_{(U,\hat{B},\hat{C},\hat{M})}\mathcal{N}$.
The following theorem gives an expression for the Riemannian gradient of $J_2(U,\hat{B},\hat{C},\hat{M})$.
\begin{theorem}
	Consider LQO systems $\eqref{eqn-1}$ and reduced models $\eqref{eqn-2}$, which are 
	determined by the product manifold $\mathcal{N}$. Let $X$, $\hat{P}$, $K$ and $L$ be 
	the solutions of the equations $\eqref{eqn-12}$, $\eqref{eqn-13}$, $\eqref{eqn-27}$ 
	and $\eqref{eqn-28}$, respectively.
	Then the Riemannian gradient of $J_2(U,\hat{B},\hat{C},\hat{M})$ is computed as
	\begin{equation}\label{eqn-51}
		\begin{split}
			&\mathrm{grad}J_2(U,\hat{B},\hat{C},\hat{M})\\
			=&2(P_U(A^\top U(X^\top K+\hat{P}L)^\top+AU(X^\top K+\hat{P}L)),
			K^\top B+L\hat{B},\hat{C}\hat{P}-CX,\hat{P}\hat{M}\hat{P}-X^\top MX).
		\end{split}
	\end{equation}
where $P_U$ is defined in \eqref{pro_U}.
\end{theorem}
\begin{proof}
	We denote the natural extension of $J_2$ from $\mathcal{N}$ to $\bar{\mathcal{N}}$ 
	as  
	$\bar{J}_2$.
	The derivative of $\bar{J}_2$ along the direction 
	$(U^\prime,B^\prime,C^\prime,M^\prime)\in T_{(U,\hat{B},\hat{C},\hat{M})}\mathcal{N}$ 
	is
	\begin{equation}\label{eqn-52}
		\begin{aligned}
			&\mathrm{D}\bar{J_2}(U,\hat{B},\hat{C},\hat{M})[(U^\prime,B^\prime,C^\prime,M^\prime)]=\\
			&2\mathrm{tr}({B^\prime}^\top(Y^\top B+\hat{Q}\hat{B}))+2\mathrm{tr}((B\hat{B}^\top)^\top\mathrm{D}Y[(U^\prime,B^\prime,C^\prime,M^\prime)])+\mathrm{tr}(\hat{B}{\hat{B}}^\top\mathrm{D}\hat{Q}[(U^\prime,B^\prime,C^\prime,M^\prime)]),
		\end{aligned}
	\end{equation}
	where $\mathrm{D}Y[(U^\prime,B^\prime,C^\prime,M^\prime)]$ and 
	$\mathrm{D}\hat{Q}[(U^\prime,B^\prime,C^\prime,M^\prime)]$ are obtained by 
	differentiating \eqref{Y_eqn} and \eqref{hat_Q_eqn}, respectively,
	\begin{align}
		A^\top\mathrm{D}Y[(U^\prime,B^\prime,C^\prime,M^\prime)]+\mathrm{D}Y[(U^\prime,B^\prime,C^\prime,M^\prime)]\hat{A}+H_1&=0,\label{eqn-53}\\
		\hat{A}^\top\mathrm{D}\hat{Q}[(U^\prime,B^\prime,C^\prime,M^\prime)]+\mathrm{D}\hat{Q}[(U^\prime,B^\prime,C^\prime,M^\prime)]\hat{A}+H_2&=0,\label{eqn-54}
	\end{align}
	in which
	\begin{align}
		H_1=&Y{U^\prime}^\top AU+YU^\top AU^\prime-C^\top C^\prime-M\mathrm{D}X[(U^\prime,B^\prime,C^\prime,M^\prime)]\hat{M}-MXM^\prime,\nonumber\\
		H_2=&{U^\prime}^\top A^\top U\hat{Q}+U^\top A^\top U^\prime\hat{Q}+\hat{Q}{U^\prime}^\top AU+\hat{Q}U^\top AU^\prime+{C^\prime}^\top\hat{C}+\hat{C}^\top C^\prime\nonumber\\
		&M^\prime\hat{P}\hat{M}+\hat{M}\mathrm{D}\hat{P}[(U^\prime,B^\prime,C^\prime,M^\prime)]
		\hat{M}+\hat{M}\hat{P}M^\prime.\nonumber
	\end{align}
	Besides, it follows from \eqref{eqn-12} and \eqref{eqn-13} that there hold
	\begin{align}
		A\mathrm{D}X[(U^\prime,B^\prime,M^\prime)]+\mathrm{D}X[(U^\prime,B^\prime,M^\prime)]\hat{A}^\top+X{U^\prime}^\top
		 A^\top U+XU^\top A^\top U^\prime+B{B^\prime}^\top&=0,\label{eqn-55}\\
		\hat{A}\mathrm{D}\hat{P}[(U^\prime,B^\prime,M^\prime)]+\mathrm{D}\hat{P}
		[(U^\prime,B^\prime,M^\prime)]{\hat{A}}^\top+{U^\prime}^\top AU\hat{P}+U^\top 
		AU^\prime\hat{P}+& \nonumber \\
		\hat{P}{U^\prime}^\top A^\top U+\hat{Q}U^\top A^\top 
		U^\prime+B^\prime{\hat{B}}^\top+\hat{B}{B^\prime}^\top&=0.\label{eqn-56}
	\end{align}

	Applying \lemmaref{lemma:1} to $\eqref{eqn-12}$ and $\eqref{eqn-53}$ leads to
	\begin{align}
		\mathrm{tr}((B\hat{B}^\top)^\top\mathrm{D}Y[(U^\prime,B^\prime,C^\prime,M^\prime)])=&\mathrm{tr}(X^\top
		 H_1)\nonumber=\mathrm{tr}({U^\prime}^\top(AUX^\top Y+A^\top UY^\top X))\\
		&-\mathrm{tr}({C^\prime}^\top CX)\nonumber-\mathrm{tr}({M^\prime}^\top X^\top 
		MX)\nonumber\\
		&-\mathrm{tr}(\mathrm{D}X[(U^\prime,B^\prime,C^\prime,M^\prime)]\hat{M}X^\top 
		M).\nonumber
	\end{align}
    It follows from $\eqref{eqn-35}$ and $\eqref{eqn-55}$ that  
	\begin{align}
		\mathrm{tr}(\mathrm{D}X[(U^\prime,B^\prime,C^\prime,M^\prime)]\hat{M}X^\top M)=&\mathrm{tr}((MX\hat{M})^\top\mathrm{D}X[(U^\prime,B^\prime,C^\prime,M^\prime)])\nonumber\\
		=&\mathrm{tr}(R_1^\top(X{U^\prime}^\top A^\top U+XU^\top A^\top U^\prime+B{B^\prime}^\top))\nonumber\\
		=&\mathrm{tr}({U^\prime}^\top(AUX^\top R_1+A^\top UR_1\top 
		X))+\mathrm{tr}({B^\prime}^\top R_1^\top B),\nonumber
	\end{align}
	which implies that 
	\begin{equation}\label{eqn-59}
		\begin{aligned}
			\mathrm{tr}((B\hat{B}^\top)^\top\mathrm{D}Y[(U^\prime,B^\prime,C^\prime,M^\prime)])=&\mathrm{tr}({U^\prime}^\top(AUX^\top K+A^\top UK^\top X))-\mathrm{tr}({B^\prime}^\top R_1^\top B)\\
			&-\mathrm{tr}({C^\prime}^\top CX)-\mathrm{tr}({M^\prime}^\top X^\top MX).
		\end{aligned}
	\end{equation}

	Similarly, after a straightforward algebraic manipulation based on $\eqref{eqn-13}$, 
	$\eqref{R_equation}$, 
	$\eqref{eqn-54}$ and $\eqref{eqn-56}$, we obtain 
	\begin{equation}\label{eqn-60}
		\begin{aligned}
			\mathrm{tr}(\mathrm{D}\hat{Q}[(U^\prime,B^\prime,C^\prime,M^\prime)]\hat{B}{\hat{B}}^\top)=&2\mathrm{tr}({U^\prime}^\top(AU\hat{P}L+A^\top UL\hat{P}))+2\mathrm{tr}({B^\prime}^\top R_2\hat{B})\\
			&+2\mathrm{tr}({C^\prime}^\top\hat{C}\hat{P})+2\mathrm{tr}({M^\prime}^\top\hat{P}\hat{M}\hat{P}),
		\end{aligned}
	\end{equation}
	where $L$ is determined by $\eqref{eqn-28}$. Substituting $\eqref{eqn-59}$ and 
	$\eqref{eqn-60}$ into $\eqref{eqn-52}$ gives
	\begin{align}
		&\mathrm{D}\bar{J_2}(U,\hat{B},\hat{C},\hat{M})[(U^\prime,B^\prime,C^\prime,M^\prime)]\nonumber\\
		=&2\mathrm{tr}({U^\prime}^\top(AU(X^\top K+\hat{P}L+A^\top U(X^\top 
		K+\hat{P}L)^\top))\nonumber\\
		&+2\mathrm{tr}({B^\prime}^\top(K^\top 
		B+L\hat{B}))+2\mathrm{tr}({C^\prime}^\top(\hat{C}\hat{P}-CX))
		+2\mathrm{tr}({M^\prime}^\top(\hat{P}\hat{M}\hat{P}-X^\top MX)),\nonumber
	\end{align}
	which leads to the Euclidean gradient
	\begin{align}
		&\mathrm{grad}\bar J_2(U,\hat{B},\hat{C},\hat{M}) \nonumber\\
		=&2(A^\top U(X^\top 
		K+\hat{P}L)^\top+AU(X^\top K+\hat{P}L),K^\top 
		B+L\hat{B},\hat{C}\hat{P}-CX,\hat{P}\hat{M}\hat{P}-X^\top MX).\nonumber
	\end{align}
	Because of 
	$\mathrm{grad}J_2=P_{(U,\hat{B},\hat{C},\hat{M})}(\mathrm{grad}\bar{J}_2)$, one can 
	get \eqref{eqn-51} based on the Euclidean gradient easily. We conclude the proof.
\end{proof}

\begin{algorithm}
	\renewcommand{\algorithmicrequire}{\textbf{Input:}}
	\renewcommand{\algorithmicensure}{\textbf{Output:}}
	\caption{Riemannian conjugate gradient method for MOR based on $\mathcal{N}$
		(PRCG).}
	\label{alg:2}
	\begin{algorithmic}
		\Require The coefficient matrices of LQO system $\Sigma$, and a positive integer 
		$k_{max}$
		\Ensure Reduced LQO model $\hat\Sigma$.
		\\Choose an initial matrix $U_0\in\mathrm{St}(n, r)$ as well as 
		$\hat{B}_0\in\mathbb{R}^{r\times m}$, $\hat{C}_0\in\mathbb{R}^{1\times r}$ and 
		$\hat{M}_0\in \mathrm{S}_r$.
		\\Compute the Riemannian gradient $\mathrm{grad}J_{2}(\mathcal{N}_0)$ by 
		$\eqref{eqn-51}$, and set $\eta_0=-\mathrm{grad}J_{2}(\mathcal{N}_0)$.
		\For{$k=0,1,\cdots, k_{max}-1$}
		\\\quad Choose a step size $t_k$ fulfilled $\eqref{eqn-64}$ and 
		$\eqref{eqn-65}$.
		\\\quad Set 
		$(U_{k+1},\hat{B}_{k+1},\hat{C}_{k+1},\hat{M}_{k+1})=\mathcal{R}_{\mathcal 
			N_k}(t_k\eta_k)$, and compute $\mathrm{grad}J_{2}(\mathcal N_{k+1})$.
		\\\quad Compute $\tilde{\mathcal{T}}_{t_k\eta_k}(\eta_k)$ and $\beta_{k+1}$ via  
		$\eqref{eqn-61}$ and $\eqref{eqn-63}$, respectively. 
		\\\quad Update the search direction via $\eta_{k+1}=-\mathrm{grad}J_{2}({\mathcal 
			N_{k+1}})+\beta_{k+1}\tilde{\mathcal{T}}_{t_{k}\eta_{k}}(\eta_{k})$.
		\EndFor
		\\\textbf{return} $\hat{A}=U^\top_{k_{max}} AU_{k_{max}}$, $\hat{B}_{k_{max}}$, 
		$\hat{C}_{k_{max}}$, $\hat{M}_{k_{max}}.$
	\end{algorithmic}  
\end{algorithm}

For solving optimization problem $\eqref{eqn-47}$, we also need to define a vector 
transport 
$\mathcal{T}$ associated with the the product manifold $\mathcal{N}$. Because of 
$\eqref{eqn-23}$ and $\eqref{eqn-48}$, one can define simply the vector transmission as 
\begin{equation}\label{eqn-61}
	\mathcal{T}_{(U_2^\prime,B_2^\prime,C_2^\prime,M_2^\prime)}(U_1^\prime,B_1^\prime,C_1^\prime,M_1^\prime)=(\mathcal{T}_{U_2^\prime}(U_1^\prime),B_1^\prime,C_1^\prime,M_1^\prime),
\end{equation}
for 
$(U_1^\prime,B_1^\prime,C_1^\prime,M_1^\prime),(U_2^\prime,B_2^\prime,C_2^\prime,M_2^\prime)\in
 T_{(U,\hat{B},\hat{C},\hat{M})}\mathcal{N}$. For brevity, we denote the metric and norm 
 on the tangent 
 space $T_{(U_k,\hat{B}_k,\hat{C}_k,\hat{M}_k)}\mathcal{N}$ by  
 $\left<\cdot,\cdot\right>_{\mathcal N_k}$ and $\left\|\cdot\right\|_{\mathcal N_k}$, 
 respectively. Besides, the retraction at the point $(U_k,\hat{B}_k,\hat{C}_k,\hat{M}_k)$ 
 is 
 referred as $\mathcal{R}_{\mathcal N_k}$, and the Riemannian gradient of $J_2$ at 
 $(U_k,\hat{B}_k,\hat{C}_k,\hat{M}_k)$ is denoted as $\mathrm{grad}J_{2}({\mathcal N_k})$.
Along the same line as \autoref{alg:1}, the iteration direction at the $k+1$ step of the 
optimization algorithm is given by
\begin{equation}
	\eta_{k}=-\mathrm{grad}J_{2}({\mathcal 
	N_k})+\beta_{k}\tilde{\mathcal{T}}_{t_{k-1}\eta_{k-1}}(\eta_{k-1}),\nonumber
\end{equation}
for $k=1,2, \cdots$, where the parameter $\beta_{k}$ is calculated according to
\begin{equation}\label{eqn-63}
	\beta_{k}=\frac{\|\mathrm{grad}J_{2}({\mathcal N_k})\|_{\mathcal 
	N_k}^2}{\langle\mathrm{grad}J_{2}({\mathcal 
	N_k}),\tilde{\mathcal{T}}_{t_{k-1}\eta_{k-1}}(\eta_{k-1})\rangle_{\mathcal 
	N_k}-\langle
	 \mathrm{grad}J_{2}({\mathcal N_k}),\eta_{k-1}\rangle_{\mathcal N_k}},
\end{equation}
and the deflated vector transport $\tilde{\mathcal{T}}_{t_{k-1}\eta_{k-1}}(\eta_{k-1})$ 
is 
constructed based on $\mathcal{T}_{t_{k-1}\eta_{k-1}}(\eta_{k-1})$ with the 
same strategy 
in \eqref{eqn-44}. The step size $t_k$ should satisfy the Wolfe conditions
\begin{equation}\label{eqn-64}
	J_2(\mathcal{R}_{\mathcal N_k}(t_k\eta_k))\leq 
	J_2(U_k,\hat{B}_k,\hat{C}_k,\hat{M}_k)+c_1t_k\langle\mathrm{grad}J_{2}({\mathcal 
	N_k}),\eta_k\rangle_{\mathcal N_k},
\end{equation}
\begin{equation}\label{eqn-65}
	\langle\mathrm{grad}J_2(\mathcal{R}_{\mathcal 
	N_k}(t_k\eta_k)),\mathcal{T}_{t_k\eta_k}(\eta_k)\rangle_{\mathcal{R}_{\mathcal 
	N_k}(t_k\eta_k)}\geq
	 c_2\langle\mathrm{grad}J_{2}({\mathcal N_k}),\eta_k\rangle_{\mathcal N_k},
\end{equation}
where $0<c_1<c_2<1$. We present the main steps of Riemannian conjugate gradient method for
\eqref{eqn-47} based on the product manifold in \autoref{alg:2}.


\section{Efficient execution via the approximate solutions of Sylvester 
equations}\label{sec:sec-4}

The focus of linear-search strategies is the choice of 
the search direction and the step size in the procedure of numerical optimization. In 
\autoref{alg:1} and \autoref{alg:2}, a couple 
of Sylvester equations are solved repeatedly in each iterate for the selection of the 
direction, where the coefficient matrices change with the iteration. Besides, the 
evaluation of the cost function in Wolf conditions is also associated with the 
solution of Sylvester equations. Because the order of $\eqref{eqn-13}$ and 
$\eqref{eqn-28}$ is $r\times r$, where $r$ is the order of reduced models, one can get 
the solution  
via the standard solvers efficiently \cite{Simoncini2016}. 
However, $\eqref{eqn-12}$ and $\eqref{eqn-27}$ are of order $n\times r$, where $n$ is the 
order of original systems and much higher, and the standard solver for the solution of 
$\eqref{eqn-12}$ and $\eqref{eqn-27}$ is numerically expensive. In fact, the expense of 
solving the high-order Sylvester equations dominates the whole cost of the proposed 
algorithms. In this section, we present an efficient scheme to get an approximate 
solution of 
$\eqref{eqn-12}$ and $\eqref{eqn-27}$, which takes advantage of the special structure of 
Sylvester equations involved in \autoref{alg:1} and \autoref{alg:2} and enables an 
efficient execution for our approach.

We start with the integral formulation of the 
Gramians. It follows from $\eqref{eqn-3}$ that the controllability Gramian $P_e$ of 
$\Sigma_e$ has the similar expression  
\begin{equation}
	P_e=\int_0^\infty e^{A_et}B_eB_e^\top e^{A_e^\top t}\mathrm{d}t.\nonumber
\end{equation}
We rewrite it as a block form
\begin{align}
	P_e=\int_0^\infty
	\left[
	\begin{matrix}
		e^{At}BB^\top e^{A^\top t}&e^{At}B\hat{B}^\top e^{\hat{A}^\top t}\\
		e^{\hat{A}t}\hat{B}B^\top e^{A^\top t}&e^{\hat{A} t}\hat{B}\hat{B}^\top 
		e^{\hat{A}^\top t}\nonumber
	\end{matrix}
	\right]
	\mathrm{d}t.\label{eqn-69}
\end{align}
Then the partitioned form $\eqref{eqn-11}$ of $P_e$ implies the explicit 
expression for the solution of \eqref{eqn-12}
\begin{equation}\label{X-int}
	X=\int_0^\infty e^{At}B(e^{\hat{A}t}\hat{B})^\top\mathrm{d}t.
\end{equation}
We aim to approximate the exponential function $e^{At}$ with its truncated expansion 
over the Laguerre function basis. With the $i$th Laguerre polynomial 
\begin{equation}
	l_{i}(t)=\frac{e^{t}}{i!}\frac{\mathrm d^{i}}{\mathrm dt^{i}}(e^{-t}t^{i}),\quad 
	i=0,1,\cdots\nonumber
\end{equation}
the scaled Laguerre function is defined as
\begin{equation}
	\phi_i^\alpha(t)=\sqrt{2\alpha}e^{-\alpha t}l_i(2\alpha t),\nonumber
\end{equation}
where $\alpha$ is a positive scaling parameter called time-scale factor 
\cite{Atkinson2005,Knockaert2000}.
The sequence $\phi_i^\alpha(t)$ of scaled Laguerre functions
forms a uniformly bounded orthonormal basis for the Hilbert space $L_2(\mathbb R_{+})$. 
For the stable matrix $A$, there holds
\begin{equation}\label{eqn-66}
	e^{At}=\sum_{i=0}^\infty A_i\phi_i^\alpha\left(t\right),
\end{equation}
where the coefficient matrices $\{A_i\}_{i=0}^\infty $  are defined by the recursive 
formula \cite{Xiao2022,Moore2011,Long2012}
\begin{equation}
	\begin{aligned}
		A_0&=\sqrt{2\alpha}(\alpha I-A)^{-1},\\
		A_i&=[(A+\alpha I)(A-\alpha I)^{-1}]A_{i-1},\quad i=1,2,\cdots.\nonumber
	\end{aligned}
\end{equation}
A similar expansion $e^{\hat At}=\sum_{i=0}^\infty \hat A_i\phi_i^\alpha\left(t\right)$ 
can be obtained by replacing $A$ with $\hat A$ in \eqref{eqn-66}. We adopt the truncated 
expansion of $e^{At}$ and $e^{\hat{A}t}$ 
simultaneously in \eqref{X-int} over the same
basis $\{\phi_{i}^{\alpha}(t)\}_{i\in\mathbb{N}}$, and the solution $X$ is approximated 
as 
\begin{align}\label{app_Lag_X}
	X&=\int_0^\infty\left(\sum_{i=0}^{\infty}A_{i}B\phi_{i}^{\alpha}\left(t\right)\right)
	\left(\sum_{j=0}^{\infty}\hat{A}_j\hat{B}\phi_{j}^{\alpha}\left(t\right)\right)^\top\mathrm{d}t\nonumber\\
	&\approx\int_0^\infty\left(\sum_{i=0}^{N-1}A_iB\phi_i^\alpha\left(t\right)\right)
	\left(\sum_{j=0}^{N-1}\hat{A}_j\hat{B}\phi_j^\alpha\left(t\right)\right)^\top\mathrm{d}t.
\end{align}
Due to the orthogonality of Laguerre functions
\begin{equation}\nonumber
	\int_0^\infty\phi_i^\alpha(t)\phi_j^\alpha\mathrm{d}t=
	\begin{cases}
		0,&i\neq j,\\
		1,&i=j,
	\end{cases}
\end{equation}
the approximation \eqref{app_Lag_X} reduces to 
\begin{equation}\nonumber
	X\approx A_0B(\hat A_0\hat B)^\top+A_1B(\hat A_1\hat B)^\top+\cdots+A_{N-1}B(\hat 
	A_{N-1}\hat B)^\top. 
\end{equation}
With the notation
\begin{equation*}
		F=\left(\begin{matrix}A_0B\quad 
		A_1B\quad\cdots\quad A_{N-1}B\end{matrix}\right),\quad
		\hat{F}=\left(\begin{matrix}\hat{A}_0\hat{B}\quad 
		\hat{A}_1\hat{B}\quad\cdots\quad\hat{A}_{N-1}\hat{B}\end{matrix}\right),
\end{equation*}
we get the low-rank approximate solution of (\ref{eqn-12})
\begin{equation}\label{eqn-70}
	X\approx F\hat{F}^\top. 
\end{equation}
As a consequence, one can calculate the approximate solution $X$ in \autoref{alg:1} and 
\autoref{alg:2} simply by the matrix-vector product. More importantly, although the 
factor 
$F$ is related to the high-order original systems, one just needs to calculate $F$ one 
time at the beginning of \autoref{alg:1} and 
\autoref{alg:2} because it is unchanged during the whole iteration. While the factor 
$\hat F$ 
changes step by step in the iteration, it is defined completely by reduced models, and 
can be calculated cheaply. So, the low-rank approximation \eqref{eqn-70} results in an 
elegant split for the solution of \eqref{eqn-12} which facilitates the execution of 
\autoref{alg:1} and \autoref{alg:1} a lot.

For the solution of $\eqref{eqn-27}$, there is a similar integral expression 
\begin{equation}\label{eqn-71}
	K=-\int_0^\infty e^{A^\top t}(C^\top\hat{C}+2MX\hat{M})e^{\hat{A}t}\mathrm{d}t.
\end{equation}
Substituting $X\approx F\hat{F}^\top$ into $\eqref{eqn-71}$ gives rise to
\begin{align}
	K&\approx-\int_0^\infty e^{A^\top t}(C^\top\hat{C}+2MF\hat{F}^\top\hat{M})e^{\hat{A}t}\mathrm{d}t\nonumber\\
	&=-\int_0^\infty e^{A^\top t}
	\begin{pmatrix}C^\top&\sqrt{2}MF\end{pmatrix}
	\begin{pmatrix}\hat{C}\\\sqrt{2}\hat{F}^\top\hat{M}\end{pmatrix}
	e^{\hat{A}t}\mathrm{d}t\nonumber\\
	&=-\int_0^\infty e^{A^\top t}\begin{pmatrix}C^\top&\sqrt{2}MF\end{pmatrix}
	\left(e^{\hat{A}^\top t}\begin{pmatrix}\hat{C}^\top&\sqrt{2}\hat{M}\hat{F}\end{pmatrix}\right)^\top
    \mathrm{d}t.\nonumber
\end{align}
We adopt the truncated expansion of $e^{At}$ and $e^{\hat{A}t}$ over the basis 
$\{\phi_{i}^{\alpha}(t)\}_{i\in\mathbb{N}}$ again, and a low-rank approximation to $K$ is 
derived 
\begin{equation}
	K\approx-G\hat{G}^\top,\nonumber
\end{equation}
where
\begin{align}
	G&=\left(
	\begin{matrix}
		A_0^\top\begin{pmatrix}C^\top&\sqrt{2}MF\end{pmatrix}
		&A_1^\top\begin{pmatrix}C^\top&\sqrt{2}MF\end{pmatrix}
		&\cdots&A_{N-1}^\top\begin{pmatrix}C^\top&\sqrt{2}MF\end{pmatrix}
	\end{matrix}
	\right),\nonumber\\
	\hat{G}&=\left(
	\begin{matrix}
		\hat{A}_0^\top\begin{pmatrix}\hat{C}^\top&\sqrt{2}\hat{M}\hat{F}\end{pmatrix}
		&\hat{A}_1^\top\begin{pmatrix}\hat{C}^\top&\sqrt{2}\hat{M}\hat{F}\end{pmatrix}
		&\cdots
		&\hat{A}_{N-1}^\top\begin{pmatrix}\hat{C}^\top&\sqrt{2}\hat{M}\hat{F}\end{pmatrix}
	\end{matrix}
	\right).\nonumber
\end{align}

The cost of solving Sylvester equations \eqref{eqn-12} and 
\eqref{eqn-27} repeatedly dominates the whole cost 
of \autoref{alg:1} and \autoref{alg:2}. In practice, one can solve \eqref{eqn-12} and 
\eqref{eqn-27} approximately in the iteration by the proposed method in this section. We 
take SRCG as an 
example and measure the main
computational cost by the 
number of floating point multiplications (flops). For the approximation to $X$, one can 
implement 
the matrix-vector product by performing 
an LU decomposition of $A-\alpha I$, instead of calculating the inverse 
directly, and the main cost of the factor $F$ is  
$O(\frac{2}{3}n^3+(3N-1)n^2m)$. While for the factor $\hat F$, it varies as the iteration 
continues, and the whole cost is $O(\frac{2}{3}r^3k_{max}+(3N-1)r^2mk_{max})$. Likewise, 
the cost for the factors of $K$ is $O((3N-1)n^2(1+Nm))$ and $O((3N-1)r^2(1+Nm)k_{max})$, 
respectively. Note that when the LU decomposition of $A-\alpha I$ is available, it can be 
shared for the calculation of $X$ and $K$. As a result, noticing that $r\ll n$ in the 
large-scale settings, the overall cost of \autoref{alg:1} is dominated by 
$O(\frac{2}{3}n^3+(3N-1)n^2((N+1)m+1))$ flops, which is independent of the number of 
iterates 
$k_{max}$. However, if the direct solver for Sylvester equations is employed in 
\autoref{alg:1}, the main cost is $O(2n^3{k_{max}})$ flops.

\section{Numerical Examples}\label{sec:sec-5}

In this section two numerical examples are used to illustrate the effectiveness of the 
proposed methods. All simulation results are obtained in Matlab (R2023a) on a laptop with 
Intel(R) Core(TM) i5-9300H processor with 2.40 GHz and 8 GB RAM. 

We adopt the Krylov subspace method to generate the initial value $V_0$  for \autoref{alg:1} 
(SRCG), and it leads to the following initial values for \autoref{alg:2} (PRCG)
\begin{equation*}
	\{U_0, \hat B_0, \hat C_0, \hat M_0\}=\{V_0, 
	V^\top_0 B, CV_0, V^\top_0 MV_0\}.
\end{equation*} 
The reduced order is referred as 
$r$ for brevity in general. The relative norm of Riemannian gradient 
$\delta=\|g_k\|/\|g_0\|<\varepsilon$ is used to terminate the iteration in SRCG and PRCG. 
Here, $g_0$ and $g_k$ denote Riemannian gradients in the initial 
and the $k$-th iteration, respectively, $\|\cdot\|$ is the norm defined on the
tangent space, and $\varepsilon$ is a small positive scalar to ensure 
sufficient decay of gradients. We compare the proposed methods with the existing methods, 
e.g., the POD method in \cite{Benner2017} and BT method in \cite{Benner2021} for LQO 
systems, in terms of the relative $H_2$ error $\|\Sigma-\hat\Sigma\|_{H_2}/\|\Sigma\|_{H_2}$. 
For the execution of POD, we simply assemble the uniformly distributed snapshots 
of the exact solution for a given input $u(t)$, and perform the SVD decomposition to 
extract $r$ dominate modes to generate reduced models.

\subsection{A synthetic example}

A synthetic example is introduced by following the same technique provided in 
\cite{Sato2016}. We construct randomly an $n\times n$ symmetric negative 
definite matrix $A_{sym}$ and an $n \times n$ skew-symmetric matrix $A_{skew}$. The 
coefficient matrix $A$ of \eqref{eqn-1} is defined as 
$A=A_{sym}+A_{skew}$, which implies that $A+A^\top<0$ and the system is stable. All 
elements of the input 
vector $B\in\mathbb{R}^{n}$ and the output vector $C\in\mathbb{R}^n$ are 1.
For the quadratic part in the output, we use the identity matrix $M=I$.
Note that the identity matrix is full-rank and cannot be well-approximated by a low-rank 
matrix. 

We first set $n=30$ simply to demonstrate the property of our methods. The reduced order 
is $r=6$ in the simulation, and $V_0$ is selected as the orthogonal 
basis of the subspace $\mathrm{K}_r(A,B)=\mathrm{span}\{B, AB, \cdots, A^{r-1}B\}$ with 
the aid of Arnoldi procedure. In \autoref{alg:1} and 
\autoref{alg:2}, we set $\omega=0.8, \gamma=1, c_1=0.25, c_2=0.95$ to select the step 
size $t_k=\omega^{m_k}\gamma$ satisfying the Wolfe conditions. The iteration proceeds 
until the relative norm of the gradients is less than $\varepsilon=1\times 
10^{-4}$.
\autoref{fig-2} depicts the convergence behavior of RCG-St and RCG-Pr. For this example,  
SRCG takes on faster convergence during the iteration, while PRCG results in a slightly 
lower value of cost function because it searches the minimum in a general 
framework. 

\begin{figure}[htbp]
	\centering
	\includegraphics[width=0.6\textwidth]{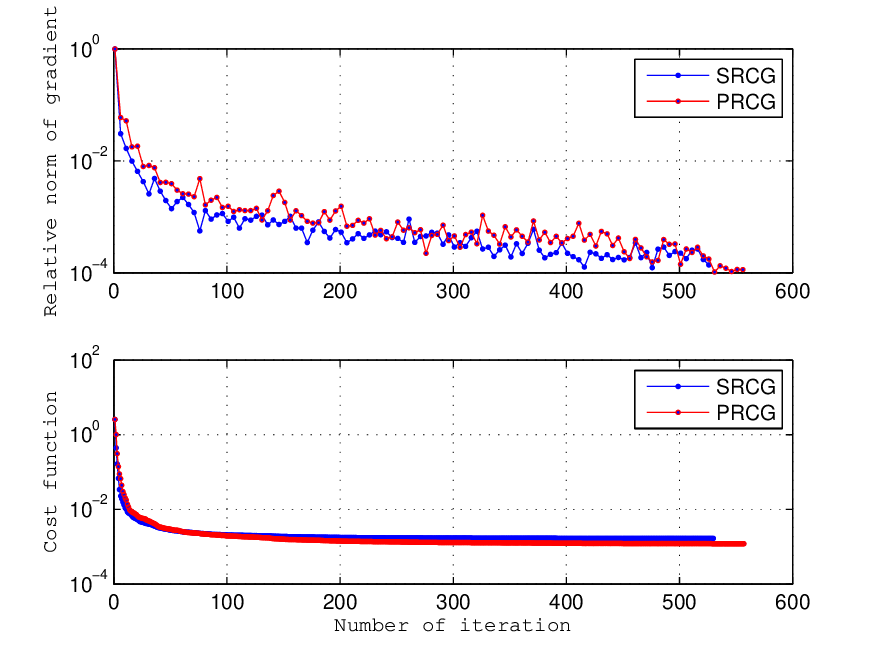}
	\caption{Evolution of the relative norm of Riemannian gradients and the value of 
		cost functions with $n=30$.}
	\label{fig-2}
\end{figure}

Given the zero initial condition and the input 
$u(t)=\mathrm{exp}(\mathrm{sin}(2t))$, \autoref{fig-3} shows the transient time responses 
and the associated relative errors of each reduced model. All reduced models are of order 
$r=6$. The "Krylov" model is the one produced by the projection matrix $V_0$, 
which exhibits a distinct mismatch with the original system in the output picture. 
Compared with the initial values, SRCG and PRCG provide much better approximation by 
minimizing the $H_2$ norm of error systems. For comparison, we solve the 
original system in the time interval [0 3], and the 
first 6 dominate modes based on 100 samples are adopted to generate the "POD" model. 
BT method presented in 
\cite{Benner2021} is also carried out in the simulation. Except for "Krylov" model, the 
other models provide almost the same accuracy in the time domain for the input 
$u(t)=\mathrm{exp}(\mathrm{sin}(2t))$. The relative $H_2$ error of
each reduced model for $r=2, 6, 10$ are listed in \autoref{tab:1}, where the POD and 
Krylov method provide relatively lower accuracy approximation.

\begin{figure}[htbp]
	\centering
	\includegraphics[width=0.6\textwidth]{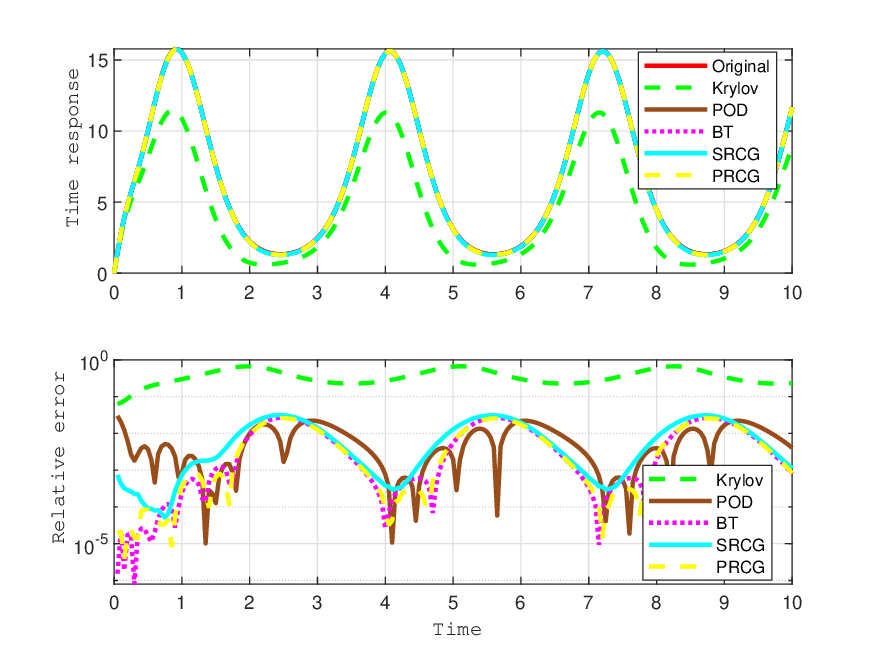}
	\caption{Time response and relative error of reduced models with the input 
	$u(t)=\mathrm{exp}(\mathrm{sin}(2t))$.}
	\label{fig-3}
\end{figure}

\begin{table}[htb]
	\centering
	\caption{\centering The relative $H_2$ norm error of reduced models for different 
		$r$ with $n=30$.}
	\label{tab:1}
	\begin{tabular}{lccccc}
		\toprule
		&Krylov&POD&BT&SRCG&PRCG\\ 
		\midrule
		$r=2$ & 6.019e-01 & 4.659e-01 & 1.635e-01 & 2.086e-01 & 1.491e-01 \\
		$r=6$ & 1.187e-01 & 7.482e-02 & 3.946e-03 & 4.798e-03 & 4.078e-03 \\
		$r=10$ & 2.054e-02 & 4.713e-03 & 9.926e-04 & 1.221e-03 & 1.017e-03 \\ 
		\bottomrule
	\end{tabular}
\end{table}

We also test the performance of the proposed methods by setting $n=300$ in this example. 
A $300 \times 300$ symmetric negative definite matrix and a $300\times 300$ 
skew-symmetric matrix are produced randomly to generate a stable matrix. The Krylov 
subspace method is executed again to obtain the initial values for SRCG and PRCG. In the 
iteration, we set $\omega=0.5, \gamma=1, c_1=0.1, c_2=0.9$, and the criterion 
$\varepsilon=1\times 10^{-3}$ is used to stop the iteration. The evolution of the norm of 
gradients and the value of cost functions is shown in \autoref{fig-12}. In 
\autoref{tab:2}, 
we list 
the relative $H_2$ error of reduced models for each value of $r=2, 6, 10$.  
Generally, the relative error of reduced models decrease notably as the reduced order 
rises.

\begin{figure}[htbp]
	\centering
	\includegraphics[width=0.6\textwidth]{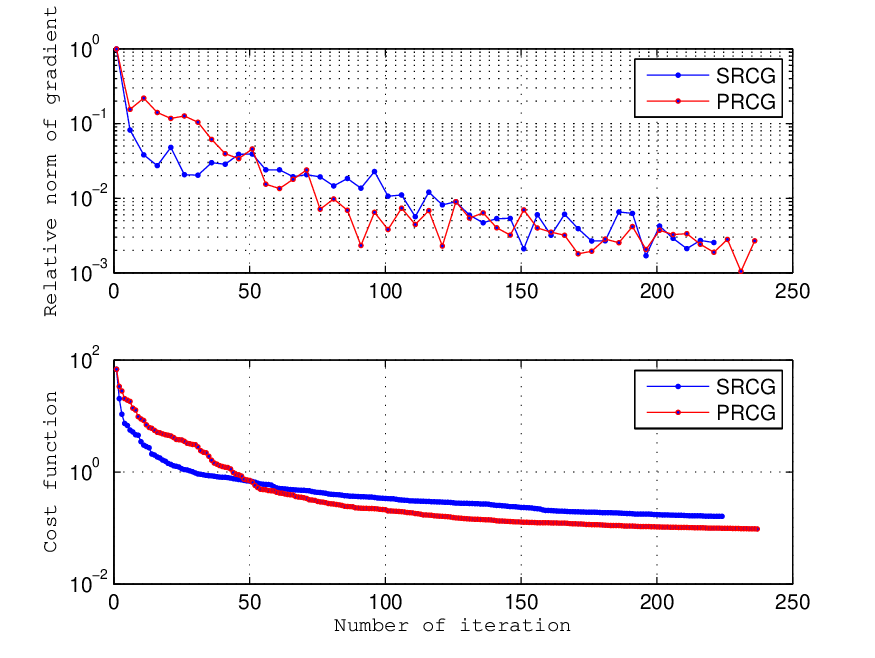}
	\caption{Evolution of the relative norm of Riemannian gradients and the value of 
		cost functions with $n=300$.}
	\label{fig-12}
\end{figure}

\begin{table}[htb]
	\centering
	\caption{\centering The relative $H_2$ error of reduced models for different 
		$r$ with $n=300$. }
	\label{tab:2}
	\begin{tabular}{lccccc}
		\toprule
		&Krylov&POD&BT&SRCG&PRCG\\ 
		\midrule
		$r=2$ & 5.605e-01 & 5.422e-01 & 8.394e-02 & 1.342e-01 & 8.304e-02  \\
		$r=6$ & 1.834e-01 & 1.642e-01 & 5.471e-03 & 1.087e-02 & 8.385e-03  \\
		$r=10$ & 5.942e-02 & 2.003e-02 & 2.493e-04 & 1.273e-03 & 1.190e-03 \\ 
		\bottomrule
	\end{tabular}
\end{table}

\subsection{Heat diffusion equation}

We consider an example coming from the collection of benchmark examples for model 
reduction \cite{Chahlaoui2002}. It is the heat diffusion equation for the one-dimensional
\begin{equation*}
	\begin{cases}
		\frac{\partial}{\partial t}T(x,t)=\alpha\frac{\partial^2}{\partial 
		x^2}T(x,t)+u(x,t),&x\in(0,1),\,t > 0,\\
		T(0,t)=T(1, t)=0,&t>0,\\
		T(x,0)=0,&x\in(0,1),
	\end{cases}
\end{equation*}
where $T(x,t)$ represents the temperature field on a thin rod. The semi-discretization of 
the spatial domain via the equidistant step size $1/(1+n)$ leads to a linear system of 
order $n$. We set the quadratic part of the output via a diagonal matrix $M$, which has 
equal diagonal elements and $\mathrm{tr}(M)=1$.

\begin{figure}[htbp]
	\centering
	\includegraphics[width=0.6\textwidth]{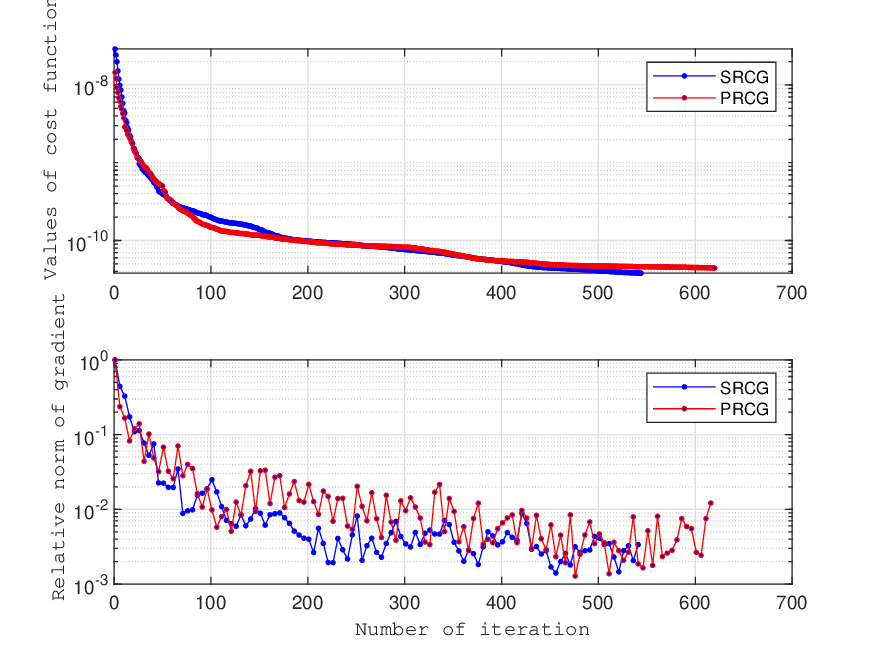}
	\caption{Evolution of the relative norm of Riemannian gradients and the value of 
		cost functions.}
	\label{fig-21}
\end{figure}

\begin{figure}[htbp]
	\centering
	\includegraphics[width=0.6\textwidth]{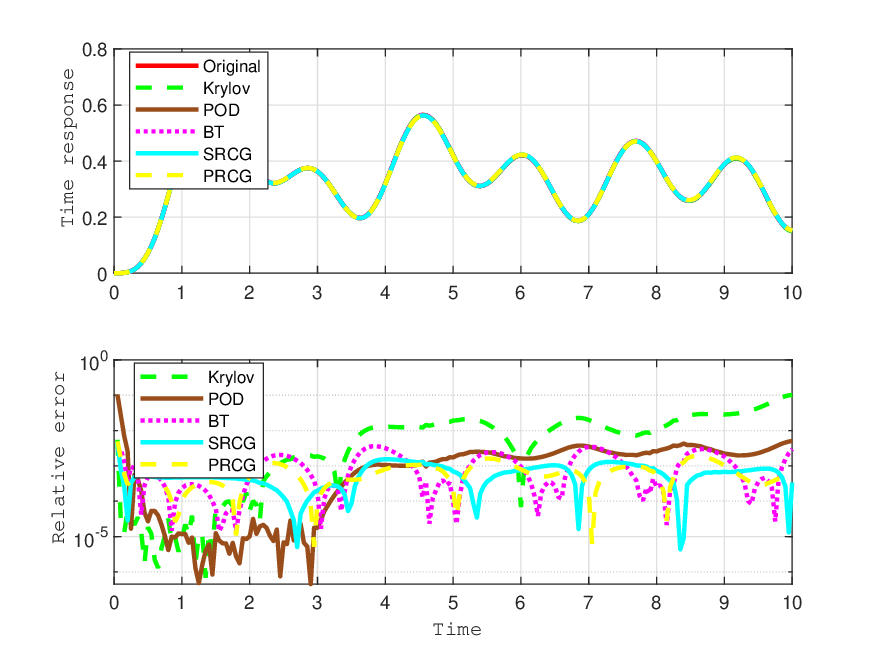}
	\caption{Time response and relative error of reduced models with the input 
		$u(t)=100\sin(2t)$.}
	\label{fig-22}
\end{figure}

We set $n=200$ and $r=10$ throughout the simulation. The initial value $V_0$ is the 
orthogonal basis of
the rational Krylov subspace spanned by the vectors $(s_iI-A)^{-1}B$, where $I$ is the identity
matrix and $s_i=2i$ for $i=2, 4, \cdots, 20$. With the parameters $\omega=0.8, \gamma=200, c_1=0.3, c_2=0.9$, \autoref{alg:1} and \autoref{alg:2} are carried out along with the criterion 
$\varepsilon=1\times 10^{-3}$. The convergence behavior of SRCG and PRCG is displayed in \autoref{fig-21},
where these two methods exhibit a similar convergence behavior in this example. We also 
plot the outputs and the corresponding relative errors of all reduced models in 
\autoref{fig-22} 
when the system is impulsed by the input $u(t)=100\sin(2t)$. Although we cannot distinct 
all reduced models clearly from the outputs, the relative error indicates that the 
"Krylov" model, generated via $V_0$, has larger error compared with the others. SRCG and 
PRCG reduce the $H_2$ 
error gradually via the iteration on the matrix manifold, and the resulting reduced 
models take on a better accuracy than the "BT" model. The "POD" model is produced by 
sampling the state on the time interval [0 3], and the error becomes larger as time 
progresses. The relative $H_2$ error of reduced models for $r=5, 10, 15$ is listed in 
\autoref{tab:3}, where SRCG, PRCG and BT yield higher accuracy in the sense of $H_2$ norm 
for this example.

\begin{table}[htb]
	\centering
	\caption{\centering The relative $H_2$ error of reduced models for different 
		$r$.}
	\label{tab:3}
	\begin{tabular}{lccccc}
		\toprule
		&Krylov&POD&BT&SRCG&PRCG\\ 
		\midrule
		$r=5$ & 1.098e-01 & 1.547e-01 & 1.046e-03 & 1.509e-02 & 1.189e-02  \\
		$r=10$ & 5.685e-03 & 4.125e-02 & 9.242e-05 & 5.922e-04 & 5.053e-04  \\
		$r=15$ & 2.423e-03 & 2.142e-02 & 1.153e-05 & 1.491e-04 & 1.112e-04  \\ 
		\bottomrule
	\end{tabular}
\end{table}

\section{Conclusions}\label{sec:sec-6}

We have investigated the $H_2$ optimal MOR for LQO systems based on the matrix manifold. 
By introducing the optimization problem on the Stiefel manifold and product manifold, the 
Dai-Yuan-type Riemannian conjugate gradient method results in the desired reduced models 
which preserve the quadratic structure of original systems. The low-rank approximation to 
the solution of Sylvester equations based on the truncated polynomial expansions enables 
an efficient execution of the proposed approach. The simulation results show that the 
proposed methods produce better reduced models in the $H_2$ norm measure.





\addcontentsline{toc}{section}{Reference}
\markboth{Reference}{}
\bibliographystyle{elsarticle-num}
\bibliography{reference}

\end{document}

%% file: default-el-journal.tex
\usepackage{xcolor}
\usepackage{tikz}
\usepackage{layouts}
\biboptions{numbers,sort&compress}
\usepackage{algorithm}
\usepackage{algpseudocode}
\usepackage{gensymb}
\usepackage[unicode=true,colorlinks=true,pdfborder={0 0 0}]{hyperref}
\hypersetup{
    bookmarks=true,         
    bookmarksopen=true,     
    pdftoolbar=true,        
    pdfmenubar=true,        
    pdffitwindow=true,      
    pdfstartview={FitH},    
    pdfnewwindow=true,      
}
\usepackage{graphicx}
\usepackage{caption}
\usepackage{subcaption}
\usepackage{wrapfig}
\usepackage{multirow,makecell}
\usepackage{longtable}
\usepackage{booktabs}
\usepackage{tabularx}
\usepackage{setspace}
\usepackage{float}
\usepackage{enumerate}
\usepackage{enumitem}

\renewcommand{\appendixname}{}
\usepackage[amsthm,thmmarks]{ntheorem}
\usepackage{siunitx}